\documentclass[letterpaper,reqno,12pt]{amsart}
\usepackage[margin=1.2in]{geometry}

\usepackage{amsmath} \usepackage{amssymb} \usepackage{amsthm}
\usepackage{amscd} 
\usepackage[all]{xy} 
\usepackage{enumerate}
\usepackage{mathrsfs}
\usepackage[dvipdfm]{hyperref}
\hypersetup{colorlinks=true}

\def\ge{\geqslant}
\def\le{\leqslant}
\def\phi{\varphi}
\def\epsilon{\varepsilon}

\def\to{\longrightarrow}

\def\mod{\operatorname{\,mod}}
\def\Hom{\operatorname{Hom}}
\def\Spec{\operatorname{Spec}}

\def\chara{\operatorname{char}}

\newcommand{\N}{\mathbb{N}}
\newcommand{\Q}{\mathbb{Q}} 
\newcommand{\C}{\mathbb{C}} 
\newcommand{\R}{\mathbb{R}} 
\newcommand{\Z}{\mathbb{Z}}

\newcommand{\A}{\mathbb{A}}
\newcommand{\T}{\mathcal{T}}
\newcommand{\sO}{\mathcal{O}}
\newcommand{\fa}{\mathfrak{a}}
\newcommand{\fb}{\mathfrak{b}}

\newcommand{\m}{\mathfrak{m}}
\newcommand{\n}{\mathfrak{n}}

\newcommand{\q}{\mathfrak{q}}
\newcommand{\U}{\mathfrak{U}}

\newcommand{\Gr}{\mathrm{Gr}}
\newcommand{\Tr}{\mathrm{Tr}}
\newcommand{\GL}{\mathrm{GL}}
\newcommand{\emb}{\mathrm{emb}}
\newcommand{\sh}{\mathrm{sh}}
\newcommand{\fjn}{\mathrm{fjn}}
\newcommand{\FJN}{\mathrm{FJN}}
\newcommand{\fpt}{\mathrm{fpt}}

\newcommand{\lct}{\mathrm{lct}}
\newcommand{\stab}{\mathrm{stab}}
\newcommand{\ustab}{\widetilde{\mathrm{stab}}}
\newcommand{\ulim}{\operatorname{ulim}}
\newcommand{\Aut}{\operatorname{Aut}}

\newcommand{\D}{\Delta}
\newcommand{\len}{\ell_R}
\newcommand{\lul}[1][I]{\ell \ell_R(R/#1)}
\newcommand{\age}[1]{\lceil #1 \rceil}
\newcommand{\qadic}[3][q]{\langle #2 \rangle_{#3, #1}}

\newcommand{\newtau}[3][e]{\tau_{#1}^{#2,#3}}

\newcommand{\fjnn}[3][e]{\fjn^{I, #2, #3}_#1}
\newcommand{\ultra}[1]{{}^* #1}
\newcommand{\cata}[1]{#1_\#}
\newcommand{\catae}[1]{[ #1 ]_m}

\newcommand{\Dreg}[3][n]{\mathcal{D}^{\mathrm{reg}}_{{#2},{#3}}}
\newcommand{\Dquot}[3][n]{\mathcal{D}^{\mathrm{quot}}_{{#2},{#3}}}

\theoremstyle{plain}
\newtheorem{thm}{Theorem}[section] 
\newtheorem{cor}[thm]{Corollary}
\newtheorem{prop}[thm]{Proposition}
\newtheorem{conj}[thm]{Conjecture}
\newtheorem*{mainthm}{Main Theorem}

\newtheorem{lem}[thm]{Lemma}
\theoremstyle{definition} 
\newtheorem{defn}[thm]{Definition}
\newtheorem{propdef}[thm]{Proposition-Definition} 
\newtheorem{eg}[thm]{Example} 
\newtheorem{obs}[thm]{Observation}
\theoremstyle{remark}
\newtheorem{rem}[thm]{Remark}
\newtheorem{ques}[thm]{Question}

\newtheorem*{cl}{Claim}
\newtheorem*{clproof}{Proof of Claim}

\newtheorem*{acknowledgement}{Acknowledgments}

\title{Ascending chain condition for $F$-pure thresholds on a fixed strongly $F$-regular germ}
\author{Kenta Sato}
\address{Graduate School of Mathematical Sciences, University of Tokyo, 3-8-1 Komaba, Meguro-ku, Tokyo 153-8914, Japan}
\email{ktsato@ms.u-tokyo.ac.jp}


\begin{document}
\tolerance = 9999

\maketitle
\markboth{KENTA SATO}{ASCENDING CHAIN CONDITION FOR $F$-PURE THRESHOLDS}

\begin{abstract}
In this paper, we prove that the set of all $F$-pure thresholds on a fixed germ of a strongly $F$-regular pair satisfies the ascending chain condition.
As a corollary, we verify the ascending chain condition for the set of all $F$-pure thresholds on smooth varieties or, more generally, on varieties with tame quotient singularities, which is an affirmative answer to a conjecture given by Blickle, Musta\c{t}\v{a} and Smith.
\end{abstract}

%%%%%%%%%%%%%%%%%%%%%%%%%%%%%%%%%%%%%%%%%%%%%%%%%%%%%%%%%%%%%%%%%%%%%%
%%%%%%%%%%%%%%%%%%%%%%%%%%%%%%%%%%%%%%%%%%%%%%%%%%%%%%%%%%%%%%%%%%%%%%
\section{Introduction}
%%%%%%%%%%%%%%%%%%%%%%%%%%%%%%%%%%%%%%%%%%%%%%%%%%%%%%%%%%%%%%%%%%%%%%
%%%%%%%%%%%%%%%%%%%%%%%%%%%%%%%%%%%%%%%%%%%%%%%%%%%%%%%%%%%%%%%%%%%%%%

In characteristic zero, Shokurov (\cite{Sho}) conjectured that the set of all log canonical thresholds on varieties of any fixed dimension satisfies the ascending chain condition.
This conjecture was partially solved by de Fernex, Ein, and Musta\c{t}\u{a} in \cite{dFEM} and \cite{dFEM2} using generic limit, and finally settled by Hacon, M\textsuperscript{c}Kernan, and Xu in \cite{HMX} using global geometry.

In this paper, we deal a positive characteristic analogue of this problem.
Let $(R,\m)$ be a Noetherian normal local ring of characteristic $p>0$ and $\D$ be an effective $\Q$-Weil divisor on $\Spec R$.
We further assume that $R$ is $F$-finite, that is, the Frobenius morphism $F:R \to R$ is a finite ring homomorphism.
For a proper ideal $\fa \subsetneq R$ and a real number $t \ge 0$, We consider the test ideal $\tau(R, \D, \fa^t)$, which is defined in terms of the Frobenius morphism (see Definition \ref{test def} below).
Since we have $\tau(R,\D, \fa^t) \subseteq \tau(R,\D, \fa^s)$ for every real numbers $0 \le s \le t$, for a given $\m$-primary ideal $I \subseteq R$, we define the $F$-jumping number of $(R,\D; \fa)$ with respect to $I$ as  
\[
\fjn^I (R, \D ;\fa) : = \inf \{ t \ge 0 \mid \tau (R, \D, \fa^t) \subseteq I \} \in \R.
\]
When $I= \m$ and $(R,\D)$ is \emph{strongly $F$-regular}, that is, $\tau(R,\D)=R$, we denote it by $\mathrm{fpt}(R,\D; \fa)$ and call it the \emph{$F$-pure threshold} of $(R,\D;\fa)$.

Since test ideals in positive characteristic enjoy several important properties which hold for multiplier ideals in characteristic zero, it is natural to ask whether or not the set of $F$-pure thresholds satisfies the ascending chain condition.
Blickle, Musta\c{t}\u{a}, and Smith conjectured the following.

\begin{conj}[\textup{\cite[Conjecture 4.4]{BMS2}}]\label{intro conj}
Fix an integer $n \ge 1$, a prime number $p>0$ and a set $\Dreg{n}{p}$ such that every element of $\Dreg{n}{p}$ is an $n$-dimensional $F$-finite Noetherian regular local ring of characteristic $p$.
The set
\[
\T^{\mathrm{reg}}_{n,p,\mathrm{pr}}: = \{ \mathrm{fpt} (A; \fa) \mid A \in \Dreg{n}{p} ,\fa \subsetneq A \textup{ is a principal ideal} \},
\]
 satisfies the ascending chain condition.
\end{conj}

This problem has been considered by several authors (\cite{BMS2}, \cite{HnBWZ}, and \cite{HnBW}).
We give an affirmative answer to this conjecture.

\begin{thm}[Corollary \ref{reg ACC}]\label{intro reg}
With the notation above, the set
\[
\T^{\mathrm{reg}}_{n,p}: = \{ \mathrm{fpt} (A; \fa) \mid A \in \Dreg{n}{p}, \fa \subsetneq A \textup{ is an ideal} \}
\]
satisfies the ascending chain condition.
\end{thm}

Employing the strategy in \cite{dFEM}, we can also verify the ascending chain condition for $F$-pure thresholds on tame quotient singularities.

\begin{thm}[Proposition \ref{quot ACC}]\label{intro quot}
Fix an integer $n \ge 1$, a prime number $p>0$ and a set $\Dquot{n}{p}$ such that every element of $\Dquot{n}{p}$ is an $n$-dimensional $F$-finite Noetherian normal local ring of characteristic $p$ with tame quotient singularities.
The set 
\[
\T^{\mathrm{quot}}_{n,p}: = \{ \mathrm{fpt} (R; \fa) \mid R \in \Dquot{n}{p}, \fa \subsetneq R \textup{ is an ideal} \}
\]
 satisfies the ascending chain condition.
\end{thm}

In order to prove Theorem \ref{intro reg}, it is enough to show that the set of all $F$-pure thresholds on a fixed $F$-finite Noetherian regular local ring satisfies the ascending chain condition.
We consider this problem in a more general setting.
Let $(R,\D)$ be a pair, that is, $(R,\m)$ is an $F$-finite Noetherian normal local ring of characteristic $p>0$ and $\D$ be an effective $\Q$-Weil divisor on $\Spec R$.
For a given $\m$-primary ideal $I \subseteq R$, we define
\[
\FJN^I(R,\D) : = \{ \fjn^I(R,\D; \fa) \mid \fa \subsetneq R \textup{ is an ideal} \} \subseteq \R_{\ge 0}.
\]
We note that if $(R,\D)$ is strongly $F$-regular and $I= \m$, then the set $\FJN^I(R,\D)$ coincides with the set of all $F$-pure thresholds 
\[
\mathrm{FPT}(R,\D) := \{ \mathrm{fpt}(R,\D; \fa) \mid \fa \subsetneq R \textup{ is an ideal} \}.
\]

\begin{mainthm}[Theorem \ref{main}]
Let $(R,\D)$ be a pair such that $K_X+\D$ is $\Q$-Cartier with index not divisible by $p$, where $K_X$ is a canonical divisor of $X=\Spec R$ and $I \subseteq R$ be an $\m$-primary ideal.
Assume that $\tau(R,\D)$ is $\m$-primary or trivial.
Then the set $\FJN^I(R,\D)$ satisfies the ascending chain condition.
In particular, if $(R,\D)$ is strongly $F$-regular, then the set $\mathrm{FPT}(R,\D)$ satisfies the ascending chain condition.
\end{mainthm}

For a real number $t>0$ and a power $q$ of $p$, we consider the ascending sequence $\{ \qadic{t}{n} \}_{n \in \N}$, where $\qadic{t}{n} : = \age{t q^n-1}/q^n$ is the \emph{$n$-th truncation} of $t$ in base $q$.
It is not so hard to prove that the set $\FJN^I(R,\D)$ satisfies the ascending chain condition if and only if for every real number $t>0$, there exists an integer $n_1>0$ with the following property:
for every ideal $\fa \subseteq R$ and every integer $n \ge n_1$, $\tau(R,\D, \fa^{\qadic{t}{n}}) \subseteq I$ if and only if $\tau(R, \D, \fa^{\qadic{t}{n_1}}) \subseteq I$.

In this paper, we define a new ideal $ \newtau{n}{u} (R,\D, \fa^t) \subseteq R$ for every integers $u,n \ge 0$ in terms of the trace map for the Frobenius morphism so that for every $n$, the sequence $\{ \newtau{n}{u}(R,\D, \fa^t) \}_{u \in \N}$ is an ascending chain which converges to $\tau(R,\D, \fa^{\qadic{t}{n}})$.
We investigate the behavior of the ideals \{$\newtau{n}{u}(R,\D, \fa^t)\}_{n \in \N}$ for some fixed $u \ge 0$ instead of the ideals $\{\tau(R, \D, \fa^{\qadic{t}{n}})\}_{n \in \N}$.
In particular, we prove the following theorem, which plays a crucial role in the proof of the main theorem.

\begin{thm}[Corollary \ref{key}]\label{intro key}
Let $(X=\Spec R, \D)$ be a pair such that $(p^e-1)(K_X+\D)$ is Cartier for some integer $e>0$, $I \subseteq R$ be an $\m$-primary ideal, $l, n_0 \ge 0$ and $u \ge 2$ be integers, and $t>0$ be a rational number such that $t=(s/p^e) + ( l/p^e(p^e-1))$ for some integers $s \ge 0$ and $0<l<p^e$.
We set $t_0 : =p^{2e}/(p^e-1)$ and $M_0=(p^{e(n_0+6)}-1) \cdot \emb(R) / (p^e-1)$, where $\emb(R)$ is the embedding dimension of $R$.
Then there exists an integer $n_1>0$ with the following property:
for any ideal $\fa \subseteq R$ such that 
\begin{enumerate}
\item $p^e>\mu_R(\fa) + \lul+\emb(R)$, where $\mu_R(\fa)$ is the number of a minimal generator of $\fa$ and $\lul := \max\{m \ge 0 \mid \m^m \subseteq I\}$, and
\item $\newtau{n_0+1}{ u}(R, \D, \fa^{l t_0}) + \m^{M_0}  \cdot \tau(R, \D) \supseteq \newtau{n_0}{ u} (R, \D, \fa^{l t_0})$
\end{enumerate}
we have
\[
\newtau{n}{u}(R,\D, \fa^t) \subseteq I \textup{ if and only if } \newtau{n_1}{u}(R,\D, \fa^t) \subseteq I
\]
for every integer $n \ge n_1$.
\end{thm}

Another key ingredient of the proof of the main theorem is the rationality of accumulation points of $\FJN^I(R,\D)$.
Blickle, Musta\c{t}\u{a}, and Smith proved in \cite{BMS2} that the set $\T^{\mathrm{reg}}_{n,p,\mathrm{pr}}$ is a closed set of rational numbers using ultraproduct.
Their proof relies on the fact that for any local ring $A \in \Dreg{n}{p}$, any principal ideal $\fa \subsetneq A$, and any integer $e \ge 0$, the test ideal $\tau(A, \fa^{1/p^e})$ can be computed by the trace map $\Tr^e: F^e_* A \to A$ for the $e$-th Frobenius morphism $F^e$, that is, we have $\tau(A, \fa^{1/p^e})= \Tr^e(F^e_* \fa)$, which fails if $\fa$ is not principal.
In order to extend the result to the non-principal case, we introduce the notion of stabilization exponent for a triple $(R,\D, \fa^t)$, which indicates how many times we should compose the trace map for the Frobenius morphism to compute the test ideal $\tau(R,\D, \fa^t)$ (see Definition \ref{stab exp}).

By combining the method used in \cite{BMS2} and some argument about the stabilization exponents, we prove the following theorem.

\begin{thm}[Theorem \ref{fpt sh2}]\label{intro BMS}
Let $(X=\Spec R, \D)$ be a pair such that $K_X+\D$ is $\Q$-Cartier with index not divisible by $p$ and $I \subseteq R$ be an $\m$-primary ideal.
Then the limit of any sequence in $\FJN^I(R,\D)$ is a rational number.
\end{thm}

As the consequence of Theorem \ref{intro key} and Theorem \ref{intro BMS}, we obtain the main theorem.

\begin{acknowledgement}
The author wishes to express his gratitude to his supervisor Professor Shunsuke Takagi for his encouragement, valuable advice and suggestions. 
The author is also grateful to Professor Mircea Musta\c{t}\u{a} for
his helpful comments and suggestions.
He would like to thank Doctor Sho Ejiri, Doctor Kentaro Ohno, Doctor Yohsuke Matsuzawa and Professor Hirom Tanaka for useful comments.
A part of this work was carried out during his visit to University of Michigan with financial support from the Program for Leading Graduate Schools, MEXT, Japan.
He was also supported by JSPS KAKENHI 17J04317.
\end{acknowledgement}
%%%%%%%%%%%%%%%%%%%%%%%%%%%%%%%%%%%%%%%%%%%%%%%%%%%%%%%%%%%%%%%%%%%%%%
%%%%%%%%%%%%%%%%%%%%%%%%%%%%%%%%%%%%%%%%%%%%%%%%%%%%%%%%%%%%%%%%%%%%%%
\section{Preliminaries}
%%%%%%%%%%%%%%%%%%%%%%%%%%%%%%%%%%%%%%%%%%%%%%%%%%%%%%%%%%%%%%%%%%%%%%
%%%%%%%%%%%%%%%%%%%%%%%%%%%%%%%%%%%%%%%%%%%%%%%%%%%%%%%%%%%%%%%%%%%%%%

%%%%%%%%%%%%%%%%%%%%%%%%%%%%%%%%%%%%%%%%%%%%%%%%%%%%%%%%%%%%%%%%%%%%%%
\subsection{Test ideals}
%%%%%%%%%%%%%%%%%%%%%%%%%%%%%%%%%%%%%%%%%%%%%%%%%%%%%%%%%%%%%%%%%%%%%%

In this subsection, we recall the definition and some basic properties of test ideals.

A ring $R$ of characteristic $p>0$ is said to be \emph{$F$-finite} if the Frobenius morphism $F: R \to R$ is a finite ring homomorphism.

Through this paper, all rings will be assumed to be $F$-finite and of characteristic $p>0$.
If $R$ is an $F$-finite Noetherian normal ring, then $R$ is excellent (\cite{Kun}) and $X=\Spec(R)$ has a \emph{canonical divisor} $K_X$ (see for example \cite[p.4]{ST17}).

\begin{defn}
A \emph{pair} $(R, \D)$ consists of an $F$-finite Noetherian normal local ring $(R, \m)$ and an effective $\Q$-Weil divisor $\D$ on $\Spec R$.
A \emph{triple} $(R, \D, \fa_\bullet^{t_\bullet} = \prod_{i=1}^m \fa_i^{t_i})$, consists of a pair $(R, \D)$ and a symbol $\fa_\bullet^{t_\bullet}= \prod_{i=1}^m \fa_i^{t_i}$, where $m>0$ is an integer, $\fa_1, \dots, \fa_m \subseteq R$ are ideals, and $t_1, \dots, t_m \ge 0$ are real numbers. 
\end{defn} 

\begin{defn}
Let $(R, \D,  \fa_\bullet^{t_\bullet} = \prod_{i=1}^m \fa_i^{t_i})$ be a triple.
An ideal $J \subseteq R$ is \emph{uniformly $(\D,  \fa_\bullet^{t_\bullet} , F)$-compatible} if 
$\phi( F^e_* (\fa_1^{\age{t_1 (p^e-1)}}  \cdots \fa_m^{\age{t_m (p^e-1)}} J)) \subseteq J$ 
for every $e \ge 0$ and every $\phi \in \Hom_R(F^e_* R(\age{(p^e-1)\D}), R)$.
\end{defn}

\begin{defn}\label{test def}
Let $(R, \D,  \fa_\bullet^{t_\bullet} =\prod_{i=1}^m \fa_i^{t_i})$ be a triple.
Assume that $\fa_1, \dots, \fa_m$ are non-zero ideals.
Then we define the \emph{test ideal} 
\[
\tau(R, \D,  \fa_\bullet^{t_\bullet} )=\tau(R,\D, \prod_{i=1}^m \fa_i^{t_i}) = \tau(R,\D, \fa_1^{t_1} \cdots \fa_m^{t_m})
\] to be an unique minimal non-zero uniformly $(\D,  \fa_\bullet^{t_\bullet}, F)$-compatible ideal.
The test ideal always exists (see \cite[Theorem 6.3]{Sch10}).

When $\fa_i=R$ and $t_i=0$ for every $i$, then we denote the ideal $\tau(R, \D, \fa_\bullet^{t_\bullet})$ by $\tau(R, \D)$.
If $\fa_i=0$ for some $i$, then we define $\tau(R, \D, \fa_\bullet^{t_\bullet})=(0)$.

\end{defn}

\begin{lem}\label{test inc}
Let $(X=\Spec R, \D, \fa^t)$ be a triple.
Then the following hold.
\begin{enumerate}
\item If $t \le t'$ and $\fa' \subseteq \fa$, then $\tau(R, \D, (\fa')^{t'}) \subseteq \tau(R, \D , \fa^t)$.
\item \textup{(\cite[Lemma 6.1]{ST14})} Assume that $K_X+ \D$ is $\Q$-Cartier. Then there exists a real number $\epsilon>0$ such that if $t \le t' \le t+ \epsilon$, then $\tau(R, \D, \fa^{t'} )= \tau(R, \D , \fa^t)$
\end{enumerate}
\end{lem}

\begin{defn}
Let $(R, \D)$ be a pair and $\fa \subseteq R$ be an ideal.
A real number $t > 0$ is called a \emph{$F$-jumping number} of $(R, \D; \fa)$ if 
\[
\tau(R, \D, \fa^{t-\epsilon}) \neq \tau(R, \D, \fa^t ),
\]
for all $\epsilon > 0$. 
\end{defn}

\begin{prop}[\textup{\cite[Theorem B]{ST14}}]\label{disc rat}
Let $(X=\Spec R, \D, \fa)$ be a triple such that $K_X+\D$ is $\Q$-Cartier.
Then the set of all $F$-jumping numbers of $(R, \D; \fa)$ is a discrete set of rational numbers.
\end{prop}

\begin{defn}
Let $(R, \D, \fa)$ be a triple such that $\fa \neq R$ and $I \subseteq R$ be an $\m$-primary ideal.
We define the \emph{$F$-jumping number} of $(R, \D; \fa)$ with respect to $I$ as
\[
\fjn^I (R, \D; \fa) := \inf \{ t \in \R_{\ge 0} \mid \tau(R, \D, \fa^t) \subseteq I  \} \in \R_{\ge 0}.
\]

When $\tau(R,\D)=R$ and $I = \m$, we denote it by $\mathrm{fpt}(R, \D ;\fa)$ and call it the \emph{$F$-pure threshold} of $(R,\D;\fa)$.
If $\D=0$, then we denote it by $\mathrm{fpt}(R; \fa)$.
\end{defn}

\begin{defn}
Let $(X=\Spec R, \D)$ be a pair and $e \ge 0$ be an integer.
Assume that $(p^e-1)(K_X+\D)$ is Cartier.
Then there exists an isomorphism
\[
\Hom_R(F^e_*( R((p^e-1)\D)) , R) \cong F^e_*R
\]
as $F^e_*R$-modules (see for example \cite[Lemma 3.1]{Sch12}).
We denote  by $\phi_\D^e$ a generator of $\Hom_R(F^e_*( R((p^e-1)\D)), R)$ as an $F^e_*R$-module.
\end{defn}

\begin{rem}
Although a map $\phi_\D^e : F^e_*R \to R$ is not uniquely determined, it is unique up to multiplication by $F^e_*R^\times$. When we consider this map, we only need the information about the image of this map.
Hence we ignore the multiplication by $F^e_* R^\times$.
\end{rem}

Let $R$ be a Noetherian ring of characteristic $p>0$, $e$ be a positive integer, and $\fa \subseteq R$ be an ideal.
Then we denote by $\fa^{[p^e]}$ the ideal generated by $\{ f^{p^e} \in R \mid f \in \fa \}.$

The following proposition seems to be well-known to experts, but difficult to find a proof in the literature.

\begin{prop}\label{test base change}
Let $(R, \m)$ and $(S, \n)$ be $F$-finite Noetherian normal local rings with residue fields $k$ and $l$, respectively.
Let $R \to S$ be a flat local homomorphism, $\D_X$ be an effective $\Q$-Weil divisor on $X=\Spec R$ and $\D_Y$ be the flat pullback of $\D_X$ to $Y=\Spec S$.
Assume that $\m S =\n$ and that the relative Frobenius morphism $F^e_{l/k} : F^e_* k \otimes_k l \to F^e_* l$ is an isomorphism for every $e \ge 0$.
Then the following hold.
\begin{enumerate}
\item The morphism $R \to S$ is a regular morphism, that is, every fiber is geometrically regular.
\item The relative Frobenius morphism $F^e_{S/R} : F^e_*R \otimes_R S \to F^e_* S$ is an isomorphism for every $e \ge 0$.
\item For every $e \ge 0$, we have 
\[
\Hom_R( F^e_* R (\age{(p^e-1) \D_X }), R) \otimes_R S \cong \Hom_S( F^e_* S (\age{(p^e-1) \D_Y }), S).
\]
\item Let $(R,\D_X, \fa_\bullet^{t_\bullet} = \prod_{i=1}^m \fa_i^{t_i} )$ be a triple.
We write $(\fa_\bullet \cdot S)^{t_\bullet} : = \prod_i (\fa_i S)^{t_i}$. 
Then we have
\[\tau(R, \D_X, \fa_\bullet^{t_\bullet}) \cdot S = \tau(S, \D_Y,  (\fa_\bullet \cdot S)^{t_\bullet} ).\]
\item If $(p^e-1)(K_X+\D_X)$ is Cartier for some $e>0$, then $(p^e-1)(K_Y+ \D_Y)$ is also Cartier and $\phi^e_{\D_Y} : F^e_* S \to S$ coincides with the morphism $\phi^e_{\D_X} \otimes_R S : F^e_* R \otimes_R S \to S$ via the isomorphism $F^e_{S/R} : F^e_* R \otimes_R S \to F^e_* S$.
\end{enumerate}
\end{prop}

\begin{proof}
Since the relative Frobenius morphism $F_{l/k} : F_* k \otimes_k l \to F_*l$ is injective, the field extension $k \subseteq l$ is separable by \cite[Theorem 26.4]{Mat}.
Then (1) follows from \cite[Theorem 28.10]{Mat} and \cite{And}.

We will prove the assertion in (2). 
Fix an integer $e \ge 0$. 
By (1), the morphism $R \to S$ is generically separable.
It follows from \cite[Theorem 26.4]{Mat} that the relative Frobenius morphism $F^e_{S/R} : F^e_* R \otimes_R S \to F^e_* S$ is injective.

We next consider the surjectivity of the map $F^e_{S/R}$.
We denote the ring $F^e_* R \otimes_R S$ by $R'$.
We consider the following commutative diagram:
\[
\xymatrix{
&F^e_* S \\
S \ar[r] \ar[ur]^{F^e_S} & R' \ar[u]_{F^e_{S/R}} \\
R \ar[r]_{F^e_R} \ar[u] & F^e_* R \ar[u]
}
\]

Since the morphisms $F^e_R : R \to F^e_* R$ and $S \to R'$ are both finite and $\n \cap R = \m$, every maximal ideal of $R'$ contains the maximal ideal $F^e_* \m$ of $F^e_*R$.
Therefore, $I : = (F^e_* \m) \cdot R' \subseteq R'$ is contained in the Jacobson radical of $R'$.
On the other hand, since the finite morphism $F^e_S : F^e_* S \to S$ factors through $F^e_{S/R}$, the morphism $F^e_{S/R}$ is also finite.
Then the morphism 
\[
F^e_{S/R} \otimes_{R'} (R'/I) : R'/I \to (F^e_*S) \otimes_{R'} (R'/I)
\]
 coincides with the relative Frobenius morphism $F^e_{l/k} : F^e_* k \otimes_k l \to F^e_*l$, and hence it is surjective.
Therefore, the map $F^e_{S/R}$ is surjective by Nakayama.

We next prove the assertion in (3).
Since $S$ is flat over $R$ and $F^e_* R (\age{(p^e-1) \D_X })$ is a finite $R$-module, we have
\[
\Hom_R( F^e_* R (\age{(p^e-1) \D_X }), R) \otimes_R S \cong \Hom_S( F^e_* R (\age{(p^e-1) \D_X }) \otimes_R S, S).
\]
By (1), the flat pullback of a prime divisor on $X$ to $Y$ is a reduced divisor. 
Therefore, the Weil divisor $\age{(p^e-1) \D_Y} $ coincides with the flat pullback of $\age{(p^e-1) \D_X }$.
It follows from (2) that $F^e_* R (\age{(p^e-1) \D_X }) \otimes_R S \cong F^e_* S (\age{p^e-1} \D_Y)$, which completes the proof of (3).

For (4), it follows from (3) that the test ideal $\tau(R, \D_X, \fa_\bullet^{t_\bullet}) \cdot S$ is uniformly $(\D_Y, (\fa_\bullet \cdot S)^{t_\bullet}, F)$-compatible and $\tau(S, \D_Y, (\fa_\bullet\cdot S)^{t_\bullet} ) \cap R$ is uniformly $(\D_X, \fa_\bullet^{t_\bullet}, F)$-compatible.
Therefore, we have 
\begin{eqnarray*}
\tau(S, \D_Y, (\fa_\bullet \cdot S)^{t_\bullet}) &\subseteq& \tau(R, \D_X, \fa_\bullet^{t_\bullet}) \cdot S \textup{ and} \\
\tau(S, \D_Y, (\fa_\bullet \cdot S)^{t_\bullet} )  \cap R &\supseteq& \tau(R, \D_X, \fa_\bullet^{t_\bullet} ),
\end{eqnarray*}
which complete the proof of (4).

For (5), we assume that $(p^e-1)(K_X+\D_X)$ is Cartier.
Since the canonical divisor $K_Y$ coincides with the flat pullback of $K_X$ (\cite[Proposition 4.1]{Aoy}, see also \cite[Lemma 45.22.1]{Sta}), the Weil divisor $(p^e-1)(K_Y+\D_Y)$ is also Cartier.
The second assertion in (5) follows from (3).
\end{proof}

Let $(R, \m)$ be a Noetherian local ring.
For a finitely generated $R$-module $M$, we denote by $\mu_R(M)$ the minimal number of generators of $M$ as an $R$-module.
We denote by $\emb(R)$ the embedding dimension $\mu_R(\m)$.
If $M$ has finite length, then we denote by $\len(M)$ the length of $M$ as an $R$-module and define 
\[\ell \ell_R (M) :=\min \{ n \ge 0 \mid \m^n M =0 \}.\] 

The following lemma is well-known to experts, but we prove it for convenience.

\begin{lem}\label{Skoda}
Let $R$ be a Noetherian ring of characteristic $p>0$, let $\fa \subseteq R$ be an ideal, and let $a, b, n$ and $e$ be non-negative integers.
\begin{enumerate}
\item If $n > p^e (\mu_R(\fa)-1)$, then we have $\fa^n = (\fa^{\age{n/p^e}-\mu_R(\fa)})^{[p^e]} \cdot \fa^{n-p^e(\age{n/p^e}-\mu_R(\fa))}$.
In particular, if $b > p^e  (\mu_R(\fa)-1)$, then we have $\fa^{a p^e +b } = (\fa^{a})^{[p^e]} \cdot \fa^b$.
\item Assume that there exist ideals $\fa_1, \dots, \fa_m \subseteq R$ and integers $M_1, \dots, M_m \ge 1$ such that $\fa=\fa_1^{M_1} + \cdots + \fa_m^{M_m}$.
Set $l : = \sum_i \mu_R(\fa_i)$.
If $n >p^e (l-1)$, then we have 
\[
\fa^n = (\fa^{\age{n/p^e}-l})^{[p^e]} \cdot \fa^{n-p^e(\age{n/p^e}-l)}.
\]
In particular, if $b > p^e ((\sum_i \mu_R(\fa_i))-1)$, then we have $\fa^{a p^e +b } = (\fa^{a})^{[p^e]} \cdot \fa^b$.
\end{enumerate}
\end{lem}

\begin{proof}
The proof of (1) is straightforward by taking a minimal generator of $\fa$.
For (2), we first consider the case when $m=1$.
If $M_1=1$, then the assertion in (2) is same as that in (1).
If $l = \mu_R(\fa_1)=1$, then the assertion holds because $\fa$ is a principal ideal.
Therefore, we may assume that $M_1 \ge 2 $ and $l \ge 2$.
In this case, it follows from (1) that
\begin{eqnarray*}
\fa^n=\fa_1^{n M_1} &=& (\fa_1^{\age{n M_1/p^e} - l})^{[p^e]} \cdot \fa_1^{n M_1 - p^e(\age{n M_1/p^e} - l)}\\
& \subseteq & (\fa_1^{M_1 (\age{n/p^e}-l)})^{[p^e]} \cdot \fa_1^{n M_1 - p^e M_1 (\age{n/p^e} -l)}\\
&=& (\fa^{\age{n/p^e}-l})^{[p^e]} \cdot \fa^{n-p^e (\age{n/p^e}-l)}.
\end{eqnarray*}

We next consider the case when $m \ge 2$.
Set $\fb_i :=\fa_i^{M_i}$ and $l_i : = \mu_R(\fa_i)$.
Then we have 
\[
\fa^n = \sum_{n_1, \dots, n_m} \prod_{i=1}^m \fb_i^{n_i},
\]
where $n_i$ runs through all non-negative integers such that $\sum_i n_i=n$.
Fix such integers $n_i$ and set $s_i : =\max \{ 0,  \age{n_i/p^e}-l_i \}$.
Then it follows from the first case that $\fb_i^{n_i}=(\fb_i^{s_i})^{[p^e]} \cdot \fb_i^{n_i-p^e s_i}$ for every integer $i$.
Therefore, we have
\begin{eqnarray*}
\prod_i \fb_i^{n_i} &=&(\prod_i \fb_i^{s_i})^{[p^e]} \cdot \prod_i \fb_i^{n_i-p^e s_i}\\
& \subseteq& (\fa^{\sum_i s_i})^{[p^e]} \cdot \fa^{\sum_i (n_i -p^e s_i)},\\
& \subseteq& (\fa^{\age{n/p^e}-l})^{[p^e]} \cdot \fa^{n-p^e (\age{n/p^e}-l)},
\end{eqnarray*}
which completes the proof of (2).
\end{proof}

%%%%%%%%%%%%%%%%%%%%%%%%%%%%%%%%%%%%%%%%%%%%%%%%%%%%%%%%%%%%%%%%%%%%%%
\subsection{Ultraproduct}
%%%%%%%%%%%%%%%%%%%%%%%%%%%%%%%%%%%%%%%%%%%%%%%%%%%%%%%%%%%%%%%%%%%%%%
In this subsection, we define the ultraproduct of a family of sets and recall some properties.
We also define the catapower of a Noetherian local ring and prove some properties.
The reader is referred to \cite{Scho} for details.

\begin{defn}
Let $\U$ be a collection of subsets of $\N$.
$\U$ is called an \emph{ultrafilter} if the following properties hold:
\begin{enumerate}
\item $\emptyset \not\in \U$.
\item For every subsets $A, B \subseteq \N$, if $A \in \U$ and $A \subseteq B$, then $B \in \U$.
\item For every subsets $A, B \subseteq \N$, if $A, B \in \U$, then $A \cap B \in \U$.
\item For every subset $A \subseteq \N$, if $A \not\in \U$, then $\N \setminus A \in \U$.
\end{enumerate}

An ultrafilter $\U$ is called \emph{non-principal} if the following holds:
\begin{enumerate}
\setcounter{enumi}{4}
\item If $A$ is a finite subset of $\N$, then $A \not\in \U$.
\end{enumerate}
\end{defn}

By Zorn's Lemma, there exists a non-principal ultrafilter.
From now on, we fix a non-principal ultrafilter $\U$.

\begin{defn}
Let $\{ T_m \}_{m \in \N}$ be a family of sets.
We define the equivalence relation $\sim$ on the set $\prod_{m \in \N} T_m$ by
\[
(a_m)_m \sim (b_m)_m \textup{ if and only if } 
\left\{ m \in \N \mid a_m=b_m \right\} \in \U.
\]
We define the \emph{ultraproduct} of $\{ T_m \}_{m \in \N}$ as
\[
\ulim_{m \in \N} T_m : = \left(\prod_{m \in \N} T_m \right) / \sim.
\]
If $T$ is a set and $T_m=T$ for all $m$, then we denote $\ulim_m T_m$ by $\ultra{T}$ and call it the \emph{ultrapower} of $T$.
\end{defn}

Let $\{ T_m \}_{m \in \N}$ be a family of sets and $a_m \in T_m$ for every $m$.
We denote by $\ulim_m a_m$ the class of $(a_m)_m$ in $\ulim_m T_m$.
Let $\{ S_m \}_m$ be another family of sets and $f_m: T_m \to S_m$ be a map for every $m$.
We can define the map 
\[
\ulim_m f_m : \ulim_m T_m \to \ulim_m S_m
\]
by sending $\ulim_m a_m \in \ulim_m T_m$ to $\ulim_m f_m(a_m) \in \ulim_m S_m$.
If $T_m =T$, $S_m=S$, and $f_m= f$ for every $m \in \N$, then we denote the map $\ulim_m f_m$ by $\ultra{f} : \ultra{T} \to \ultra{S}$.

Let $\{ R_m \}_{m \in \N}$ be a family of rings and $M_m$ be an $R_m$-module for every $m$.
Then $\ulim_m R_m$ has the ring structure induced by that of $\prod_m R_m$ and $\ulim_m M_m$ has the structure of $\ulim R_m$-module induced by the structure of $\prod_m R_m$-module on $\prod_m M_m$.
Moreover, if $k_m$ is a field for every $m$, then $\ulim_m k_m$ is a field.

\begin{prop}\label{ultra field ext}
We have the following properties.
\begin{enumerate}
\item Let $R$ be a Noetherian ring and $M$ be a finitely generated $R$-module.
Then we have $\ultra{M} \cong M \otimes_R \ultra{R}$
\item Let $k$ be an $F$-finite field of positive characteristic.
Then the relative Frobenius morphism $F^e_* (k) \otimes_k \ultra{k} \to F^e_* (\ultra{k})$ is an isomorphism.
In particular, $\ultra{k}$ is an $F$-finite field. 
\end{enumerate}
\end{prop}

\begin{proof}
For (1), we consider the natural homomorphism $M \otimes_R \ultra{R} \to \ultra{M}$.
Since the functors $\ultra{(-)}$ and $(-) \otimes_R \ultra{R}$ are both right exact, we may assume that $M$ is a free $R$-module of finite rank.
In this case, the assertion is obvious.

For (2), we consider the natural bijection $\ultra{(F^e_* k)} \cong F^e_*(\ultra{k})$.
Combining with (1), the relative Frobenius morphism $F^e_*(k) \otimes_k \ultra{k} \to F^e_*( \ultra{k})$ is an isomorphism. 
\end{proof}
%%%%%%%%%%%%%%%%%%%%%%%%%%%%%%%%%%%%%%

Let $\fa_m \subseteq R_m$ be an ideal for every $m$.
Then the natural map $\ulim_m \fa_m \to \ulim_m R_m$ is injective, and hence we can consider $\ulim_m \fa_m$ as an ideal of the ring $\ulim_m R_m$.
Let $\fb_m \subseteq R_m$ be another ideals.
Then $\ulim_m \fb_m \subseteq \ulim_m \fa_m$ if and only if
\[
\left\{ m \in \N \mid \fb_m \subseteq \fa_m \right\} \in \U.
\]

Moreover, we have the equation
\[
(\ulim_m \fa_m) + (\ulim_m \fb_m) = \ulim_m (\fa_m +\fb_m).
\]

\begin{lem}\label{ulim prod}
Let $\{ R_m \}_{m \in \N}$ be a family of rings, $\fa_m, \fb_m \subseteq R_m$ be ideals for every $m$.
Assume that there exists an integer $l>0$ such that $\mu(\fa_m) \le l$ for every $m$.
Then we have 
\[
(\ulim_m \fa_m) \cdot (\ulim_m \fb_m) = \ulim_m (\fa_m \cdot \fb_m).
\]
\end{lem}

\begin{proof}
Let $\alpha =\ulim_m a_m \in \ulim_m \fa_m$ and $\beta= \ulim_m b_m \in \ulim_m \fb$.
Then we have $\alpha \cdot \beta = \ulim_m (a_m b_m) \in \ulim_m (\fa_m \cdot \fb_m)$.
This shows the inclusion $(\ulim_m \fa_m) \cdot (\ulim_m \fb_m) \subseteq \ulim_m (\fa_m \cdot \fb_m)$.

We consider the converse inclusion.
By the assumption, there exist $f_{m,1}, \dots , f_{m,l} \in \fa_m$ such that $\fa_m=(f_{m,1}, \dots, f_{m,l})$.
Then we have $\fa_m \cdot \fb_m = \sum_i f_{m, i} \cdot \fb_m$, and hence we have 
\[
\ulim_m (\fa_m \cdot \fb_m)= \sum_i f_{\infty, i} \cdot (\ulim_m \fb_m),
\] 
where $f_{\infty, i} : = \ulim_m f_{m, i} \in \ulim_m \fa_m$ for every $i$, which complete the proof of the lemma.
\end{proof}

\begin{propdef}[\textup{\cite[Theorem 5.6.1]{Gol}}]
Let $\{a_m \}_{m \in \N}$ be a sequence of real numbers such that there exist real numbers $M_1, M_2$ which satisfies $M_1<a_m<M_2$ for every $m \in \N$.
Then there exists an unique real number $w \in \R$ such that for every real number $\epsilon >0$, we have
\[
\{ m \in \N \mid |w-a_m| <\epsilon \} \in \U.
\]
We denote this number $w$ by $\sh( \ulim_m a_m)$ and call it the \emph{shadow} of $\ulim_m a_m \in \ultra{\R}$.
\end{propdef}

Let $(R, \m, k)$ be a local ring.
Then, one can show that $( \ultra{R}, \ultra{\m}, \ultra{k})$ is a local ring.
However, even if $R$ is Noetherian, the ultrapower $\ultra{R}$ may not be Noetherian because we do not have the equation $\cap_{n \in \N} (\ultra{\m})^n = 0$ in general.

\begin{defn}[\textup{\cite{Scho}}]
Let $(R, \m)$ be a Noetherian local ring and $( \ultra{R}, \ultra{\m})$ be the ultrapower.
We define the \emph{catapower} $\cata{R}$ as the quotient ring
\[
\cata{R} : = \ultra{R}/ (\cap_{n} (\ultra{\m})^n).
\]
\end{defn}

\begin{prop}[\textup{\cite[Theorem 8.1.19]{Scho}}]
Let $(R, \m, k)$ be a Noetherian local ring of equicharacteristic and $\widehat{R}$ be the $\m$-adic completion of $R$.
We fix a coefficient field $k \subseteq \widehat{R}$.
Then we have 
\[
\cata{R} \cong \widehat{R} \ \widehat{\otimes}_k (\ultra{k}).
\]
In particular, if $(R,\m)$ is an $F$-finite Noetherian normal local ring, then so is $\cata{R}$.
\end{prop}

Let $(R, \m)$ be a Noetherian local ring, $\cata{R}$ be the catapower and $a_m \in R$ for every $m$.
We denote by $\catae{a_m} \in \cata{R}$ the image of $\ulim_m a_m \in \ultra{R}$ by the natural projection $\ultra{R} \to \cata{R}$.
Let $\fa_m \subseteq R$ be an ideal for every $m \in \N$.
We denote by $[ \fa_m]_m \subseteq \cata{R}$ the image of the ideal $\ulim_m \fa_m \subseteq \ultra{R}$ by the projection $\ultra{R} \to \cata{R}$.

\begin{lem}\label{ultra incl}
Let $(R, \m)$ be a Noetherian local ring, $\fa_m, \fb_m \subseteq R$ be ideals for every $m \in \N$.
If we have $[\fa_m]_m \subseteq [\fb_m]_m$, then for every $\m$-primary ideal $\q \subseteq R$, we have
\[
 \{ m \in \N \mid \fa_m \subseteq \fb_m + \q \} \in \U.
\]
\end{lem}

\begin{proof}
By the definition of the catapower, if $[\fa_m]_m \subseteq [\fb_m]_m$, then we have
\[
\ulim_m \fa_m \subseteq \ulim_m \fb_m+ (\ultra{\m})^n.
\]
 for every $n$.

On the other hand, it follows from Lemma \ref{ulim prod} that $(\ultra{\m})^n = \ultra{(\m^n)}$.
Therefore we have 
\begin{eqnarray*}
\ulim \fa_m & \subseteq & (\ulim \fb_m) +\ultra{(\m^n)} \\
& =& \ulim (\fb_m + \m^n),
\end{eqnarray*}
which is equivalent to 
\[
\{ m \in \N \mid \fa_m \subseteq \fb + \m^n \} \in \U.
\]
This implies the assertion in the lemma.
\end{proof}

%%%%%%%%%%%%%%%%%%%%%%%%%%%%%%%%%%%%%%%%%%%%%%%%%%%%%%%%%%%%%%%%%%%%%%
%%%%%%%%%%%%%%%%%%%%%%%%%%%%%%%%%%%%%%%%%%%%%%%%%%%%%%%%%%%%%%%%%%%%%%
\section{Variants of test ideals}
%%%%%%%%%%%%%%%%%%%%%%%%%%%%%%%%%%%%%%%%%%%%%%%%%%%%%%%%%%%%%%%%%%%%%%
%%%%%%%%%%%%%%%%%%%%%%%%%%%%%%%%%%%%%%%%%%%%%%%%%%%%%%%%%%%%%%%%%%%%%%
In this section, we introduce some variants of test ideals by using the trace maps for the Frobenius morphisms and the $q$-adic expansion of a real number (Definition \ref{variants1} and \ref{variants2}).
We also introduce the stabilization exponent (Definition \ref{stab exp}).

\begin{defn}[\textup{cf. \cite[Definition 2.1, 2.2]{HnBWZ}}]
Let $q \ge 2$ be an integer, $t>0$ be a real number and $n \in \Z$ be an integer.
We define the \emph{$n$-th digit} of $t$ in base $q$ by
\[
t^{(n)} : =\age{t q^n -1 } - q \age{t q^{n-1} -1} \in \Z.
\]
We define the \emph{$n$-th round up} and the \emph{$n$-th truncation} of $t$ in base $q$ by 
\begin{eqnarray*}
\langle t \rangle ^{n, q} &: =& \age{t q^n }/q^n \in \Q \textup{, and}\\
\qadic{t}{n} &: =& \age{ t q^n -1} / q^n \in \Q,
\end{eqnarray*}
respectively.
\end{defn}

\begin{lem}\label{qadic}
Let $q \ge 2$ be an integer, $t>0$ be a real number and $n \in \Z$ be an integer.
Then the following hold.
\begin{enumerate}
\item $0 \le t^{(n)} <q$.
\item $t^{(n)}$ is eventually zero for $n \ll 0$ and is not eventually zero for $n \gg 0$.
\item $t = \sum_{m \in \Z} t^{(m)} \cdot q^{-m}$.
\item $\qadic{t}{n} = \sum_{m \le n} t^{(m)} \cdot q^{-m}$.
\item The sequence $\{ \langle t \rangle^{n,q} \}_{n \in \Z}$ is a descending chain which convergences to $t$.
\item The sequence $\{ \qadic{t}{n} \}_{n \in \Z}$ is an ascending chain which converges to $t$. 
\end{enumerate}
\end{lem}

\begin{proof}
These all follow easily from the definitions.
For the assertion in (2), we note that if $t=s/q^m$ for some integers $s$ and $m$, then we have $t^{(n)}=q-1$ for all $n > m$.
\end{proof}

\begin{defn}\label{variants1}
Let $(X=\Spec R,\D, \fa_\bullet^{t_\bullet}=\prod_i \fa_i^{t_i})$ be a triple such that $t_i>0$ for all $i$ and $e>0$ be an integer such that $(p^e-1)(K_X + \D)$ is Cartier.
For every integer $n \ge 0$, we define
\begin{eqnarray*}
\tau^{en}_{+} (R, \D, \fa_\bullet^{t_\bullet}) &: =& \phi^{en}_\D (F^{en}_*(\fa_1^{\age{t_1 p^{en}}} \cdots \fa_m^{\age{t_m p^{en}}} \cdot \tau(R, \D) )) \subseteq R \textup{ and} \\
\tau^{en}_{-} (R, \D, \fa_\bullet^{t_\bullet}) &: =& \phi^{en}_\D (F^{en}_*(\fa_1^{\age{t_1 p^{en} -1 }} \cdots \fa_m^{\age{t_m p^{en} -1}}\cdot \tau(R, \D) )) \subseteq R.
\end{eqnarray*}
\end{defn}

\begin{eg}
Let $(X=\Spec R, \D, \fa^t)$ be a triple such that $t>0$ and that $\fa$ is a principal ideal and let $e$ be a positive integer such that $(p^e-1)(K_X+\D)$ is Cartier.
Then it follows from \cite[Lemma 5.4]{BSTZ} that 
\begin{eqnarray*}
\tau^{en}_{+}(R, \D, \fa^t) &=& \tau(R, \D, \fa^{\langle t \rangle^{n,q}}) \textup{, and}\\
\tau^{en}_{-}(R,\D, \fa^t)&=&\tau(R, \D, \fa^{\qadic{t}{n}}).
\end{eqnarray*}

By Proposition \ref{disc rat}, the sequence $\{ \tau_{+}^{en} (R, \D, \fa^t) \}_n$ is an ascending chain of ideals which converges to $\tau(R, \D, \fa^t)$ and the sequence $\{ \tau_{-}^{en} (R, \D, \fa^t) \}_n$ is a descending chain of ideals which eventually stabilizes.
\end{eg}

\begin{prop}[basic properties]
Let $(R, \D, \fa_\bullet^{t_\bullet})$ and $e$ be as in Definition \ref{variants1}.
Then the following hold.
\label{lower test basic}
\begin{enumerate}
\item \textup{(\cite[Lemma 3.21]{BSTZ})} The sequence $\{ \tau^{en}_{+}(R, \D, \fa_\bullet^{t_\bullet}) \}_{n \ge 0}$ is an ascending chain which converges to the test ideal $\tau(R, \D, \fa_\bullet^{t_\bullet})$.
\item If $t_1 > 1$, then we have 
	\[\tau^{en}_{+}(R, \D, \fa_1^{t_1} \cdots \fa_m^{t_m}) \supseteq \fa_1 \cdot \tau^{en}_{+}(R, \D, \fa_1^{t_1-1} \cdots \fa_m^{t_m}).\]
Moreover, if $t_1 > \mu_R(\fa_1)$, then we have 
	\[\tau^{en}_{+}(R, \D, \fa_1^{t_1} \cdots \fa_m^{t_m}) = \fa_1 \cdot \tau^{en}_{+}(R, \D, \fa_1^{t_1-1} \cdots \fa_m^{t_m}).\]
\item $\phi^e_\D(F^e_*(\tau^{en}_{+}(R, \D, \fa_\bullet^{p^e \cdot t_\bullet} )))=\tau^{e(n+1)}_{+}(R, \D, \fa_\bullet^{t_\bullet})$, where we set $\fa_\bullet^{p^e \cdot t_\bullet} : = \prod_i \fa_i^{p^e t_i}$.
\end{enumerate}
\end{prop}

\begin{proof}
The proof of (1) follows as in the case when $m=1$, see \cite[Lemma 3.21]{BSTZ}.
If $t_1>\mu_R(\fa_1)$, then by Lemma \ref{Skoda} (1), we have $\fa_1^{\age{t_1 p^{en}} }=\fa_1^{[p^{en}]} \cdot \fa_1^{\age{(t_1-1) p^{en}}}$, which proves (2).
The assertion in (3) follows from the fact that $\phi^{e(n+1)}_\D=\phi^e_{\D} \circ F^e_* \phi^{en}_{\D}$ (\cite[Theorem 3.11 (e)]{Sch12}).
\end{proof}

\begin{defn}\label{stab exp}
Let $(R, \D, \fa_\bullet^{t_\bullet})$ and $e$ be as in Definition \ref{variants1}.
We define the \emph{stabilization exponent} of $(R, \D, \fa_\bullet^{t_\bullet}; e)$ by
\[
\stab(R, \D, \fa_\bullet^{t_\bullet} ; e) : = \min \{ n \ge 0 \mid \tau^{en}_{+}(R, \D, \fa_\bullet^{t_\bullet})= \tau(R, \D, \fa_\bullet^{t_\bullet}) \}.
\]
\end{defn}

\begin{prop}[basic properties]\label{stab basic}
Let $(R, \D, \fa_\bullet^{t_\bullet}= \prod_{i=1}^m \fa_i^{t_i} )$ and $e$ be as in Definition \ref{variants1}.
Then the following hold.
\begin{enumerate}
\item If $t_1 > \mu_R(\fa_1)$, then we have 
\[
\stab(R, \D, \fa_1^{t_1} \cdots \fa_m^{t_m}; e ) \le \stab(R, \D, \fa_1^{t_1-1} \cdots \fa_m^{t_m}; e ).\]
\item We have 
\[
\stab(R, \D, \fa_\bullet^{t_\bullet} ; e ) \le \stab(R, \D, \fa_\bullet^{p^e \cdot t_\bullet} ;e)+1.\]
\item If $t_i >\mu_R(\fa_i)$ and $(p^e-1) t_i \in \N$ for every $i$, then for any integer $n \ge 0$, the inequality $n \ge \stab(R, \D, \fa_\bullet^{t_\bullet} ; e )$ holds if and only if 
\[
\tau^{en}_{+}(R, \D, \fa_\bullet^{t_\bullet} )= \tau^{e(n+1)}_{+}(R, \D, \fa_\bullet^{t_\bullet}).
\]
\end{enumerate}
\end{prop}

\begin{proof}
The assertions in (1) and (2) follow from Proposition \ref{lower test basic} (2) and (3), respectively.

For (3), it follows from Proposition \ref{lower test basic} (2) and (3) that
\begin{eqnarray*}
\tau^{e(n+1)}_{+}( R, \D, \fa_\bullet^{t_\bullet}) &=& \phi_\D^e(F^e_*(\tau^{en}_{+}(R, \D,\fa_1^{p^e t_1} \cdots \fa_m^{p^e t_m}) ))\\
&=& \phi^e_\D(F^e_*(\fa_1^{(p^e-1) t_1} \cdots \fa_m^{(p^e-1) t_m} \cdot \tau^{en}_{+}(R, \D, \fa_\bullet^{t_\bullet}))). 
\end{eqnarray*}
Therefore, if $\tau^{en}_{+}(R, \D, \fa_\bullet^{t_\bullet} )= \tau^{e(n+1)}_{+}(R, \D, \fa_\bullet^{t_\bullet})$, then we have $\tau^{e(n+1)}_{+}(R, \D, \fa_\bullet^{t_\bullet})= \tau^{e(n+2)}_{+}(R, \D, \fa_\bullet^{t_\bullet})$, which completes the proof.
\end{proof}

\begin{prop}\label{uniform stab exp}
Let $(X=\Spec R, \D, \fa_\bullet = \prod_i \fa_i)$ be a triple, $e$ be a positive integer such that $(p^e-1)(K_X+\D)$ is Cartier.
We define
\begin{equation*}
\ustab(R, \D, \fa_\bullet ; e): = \sup_{t_1,\dots, t_m} \{ \stab(R, \D, \fa_\bullet^{t_\bullet} ; e)\},
\end{equation*}
where every $t_i$ runs through all positive rational numbers such that $(p^e-1)t_i \in \N$.
Then we have $\ustab(R, \D, \fa_\bullet ;e) < \infty$.
Moreover, for every integer $l \ge 0$ and rational numbers $t_1, \dots, t_m>0$ such that $p^{el}(p^e-1) t_i \in \N$, we have 
\[
\stab(R, \D, \fa_\bullet^{t_\bullet} ;e) \le \ustab(R, \D, \fa_\bullet ; e ) + l.
\]
\end{prop}

\begin{proof}
By Proposition \ref{stab basic} (1), we have
\begin{equation*}
\ustab(R, \D, \fa_\bullet ; e )= \sup_{t_1, \dots, t_m} \{ \stab(R, \D, \fa_\bullet^{t_\bullet} ; e )  \},
\end{equation*}
where every $t_i$ runs through all positive rational numbers such that $(p^e-1)t_i \in \N$ and $t_i \le \mu_R(\fa_i)$.
Hence we have $\ustab(R, \D, \fa_\bullet ; e )<\infty$.

The second statement follows from Proposition \ref{stab basic} (2).
\end{proof}

We next consider the sequence of ideals $\{\tau^{en}_{-} (R, \D, \fa_\bullet^{t_\bullet}) \}_n$.
In general, the sequence $\{\tau^{en}_{-} (R, \D, \fa_\bullet^{t_\bullet}) \}_n$ may not be a descending chain.
In order to make a descending chain, we mix the definitions of $\tau_+$ and $\tau_-$, and define the new variants of test ideals as below.
In fact, we later see that we can make a descending chain by using these ideals under some mild assumptions (Proposition \ref{upper test dc}).

\begin{defn}\label{variants2}
Let $(R, \D, \fa_\bullet^{t_\bullet}= \prod_i \fa_i^{t_i} )$ and $e$ be as in Definition \ref{variants1}, $\q \subseteq R$ be an ideal, and $n , u \ge 0$ be integers.
We define
\[
\tau^{n,u}_{e,\q}(R,\D, \fa_\bullet^{t_\bullet} ) : = \phi^{e(n+u)}_\D (F^{e(n+u)}_*(\fa_1^{p^{eu} \age{t_1 p^{en}-1}} \cdots \fa_m^{p^{eu} \age{t_m p^{en}-1}} \cdot \q) ). 
\]
When $\q = \tau(R,\D)$, we denote it by $\newtau{n}{u} (R, \D, \fa_\bullet^{t_\bullet} )$. 
\end{defn}

\begin{prop}[basic properties]\label{upper test basic}
Let $(X= \Spec R, \D, \fa_\bullet^{t_\bullet} = \prod_{i=1}^m \fa_i^{t_i})$ be a triple such that $t_i>0$ for every $i$ and $(q-1)(K_X+\D)$ is Cartier for some $q=p^e$, $\q \subseteq R$ be an ideal and $n,u \ge 0 $ be integers.
Then the following hold.
\begin{enumerate}
\item For real numbers $0<s_i \le t_i$, we have $\tau^{n,u}_{e,\q}(R, \D, \fa_\bullet^{s_\bullet} ) \supseteq \tau^{n,u}_{e,\q}(R, \D, \fa_\bullet^{t_\bullet}).$
Moreover, if $\qadic{t_i}{n} < s_i \le t_i$ for every $i$, then we have $\tau^{n,u}_{e,\q}(R, \D, \fa_\bullet^{s_\bullet}) = \tau^{n,u}_{e,\q}(R, \D, \fa_\bullet^{t_\bullet} ) $.
\item For ideals $\fb_i \subseteq \fa_i$ and $\q' \subseteq \q$, we have
$\tau^{n,u}_{e,\q'}(R, \D, \fb_\bullet^{t_\bullet}) \subseteq \tau^{n,u}_{e,\q}(R, \D, \fa_\bullet^{t_\bullet})$.
\item If $\fa_1 \equiv \fb_1 \mod J$ for some ideal $J$ and $\fa_i=\fb_i$ for every $i \ge 2$, then we have
\[
\tau^{n,u}_{e,\q}(R,\D, \fa_\bullet^{t_\bullet}) \equiv \tau^{n,u}_{e,\q}(R,\D, \fb_\bullet^{t_\bullet} ) \mod{\tau^{n,u}_{e,J \cdot \q} (R, \D, \prod_{i=2}^m \fa_i^{t_i})}.
\]
If $\q \equiv \q' \mod{J}$ for some ideals $\q'$ and $J$, then we have 
\[
\tau^{n,u}_{e,\q}(R,\D, \fa_\bullet^{t_\bullet}) \equiv \tau^{n,u}_{e,\q'}(R,\D, \fa_\bullet^{t_\bullet}) \mod{ \tau^{n,u}_{e,J}(R,\D, \fa_\bullet^{t_\bullet})}.
\]
\item If $\q= \fa_{m+1}^{q^u \age{t_{m+1} q^n-1} } \tau(R,\D)$, then we have $\tau^{n,u}_{e,\q}(R,\D, \fa_\bullet^{t_\bullet} ) = \tau^{n,u}_{e}(R,\D, \prod_{i=1}^{m+1} \fa_i^{t_i})$.
\item If $t_1 > 1$, then we have $\tau^{n,u}_{e,\q}(R,\D, \fa_\bullet^{t_\bullet}) \supseteq \fa_1 \cdot \tau^{n,u}_{e,\q}(R,\D, \fa_1^{t_1-1} \cdots \fa_m^{t_m})$.
Moreover, if $t_1 > \mu_R(\fa)+(1/q^n)$, then we have 
\[
\tau^{n,u}_{e,\q}(R,\D, \fa_\bullet^{t_\bullet}) = \fa_1 \cdot \tau^{n,u}_{e,\q}(R,\D, \fa_1^{t_1-1} \cdots \fa_m^{t_m}).
\]
\item $\phi^e_\D (F^e_*(\tau^{n,u}_{e,\q}(R,\D, \fa_\bullet^{p^e \cdot t_\bullet}) )) =\tau^{n+1,u}_{e,\q}(R,\D, \fa_\bullet^{t_\bullet})$. 
\item The sequence $\{ \tau^{n,u}_{e}(R,\D, \fa_\bullet^{t_\bullet}) \}_{u \in \N}$ is an ascending chain of ideals which converges to $\tau(R, \D, \prod_i \fa_i^{\qadic{t_i}{n}})$.
\item If $u \ge \ustab(R, \D, \fa_\bullet ; e)$, then we have 
\[
\newtau{n}{u} (R, \D,  \fa_\bullet^{t_\bullet}) =\tau(R, \D,  \prod_i \fa_i^{\qadic{t_i}{n}})
\]
 for every $n$.
\item Assume that $q^{u-1} \ge \mu_R(\fa_i)$ and the $n$-th digit ${t_i}^{(n)}$ of $t_i$ in base $q$ is non-zero for every $i$.
Then we have $\tau^{n,u}_{e,\q}(R, \D, \fa_\bullet^{t_\bullet})= \tau^{n-1, u}_{e, \q'}(R,\D, \fa_\bullet^{t_\bullet})$, where $\q' : = \phi^e_\D(F^e_*(\prod_i \fa_i^{q^{u} \cdot {t_i}^{(n)}} \q))$.
\end{enumerate}
\end{prop}

\begin{proof}
The assertions in (1), (2), (3), (4) and (8) follow easily from the definitions.
The assertions in (5), (6) and (7) follow from Proposition \ref{lower test basic}.
The assertion in (9) follows from Lemma \ref{Skoda} (1).
\end{proof}

\begin{prop}\label{upper test dc}
Let $(X=\Spec R,\D, \fa_\bullet^{t_\bullet})$ be a triple such that $t_i>0$ for every $i$ and $(q-1)(K_X+\D)$ is Cartier for some $q=p^e$, and $u>0$ be an integer such that $q^{u-1} \ge \max_i \mu_R(\fa_i)$.
Assume that $q(q-1) t_i \in \N$ for every $i$.
Then the sequence $\{ \newtau{n}{u}(R, \D,  \fa_\bullet^{t_\bullet} ) \}_{n \ge 1}$ is a descending chain of ideals.
\end{prop}

\begin{proof}
Since $q(q-1) t_i \in \N$, the $n$-th digit $t_i^{(n)}$ of $t_i$ in base $q$ is constant for $n \ge 2$.
By Lemma \ref{qadic} (2), it is non-zero.
Therefore, the assertion follows from Proposition \ref{upper test basic} (2) and (9).
\end{proof}

\begin{defn}
Let $(X=\Spec R, \D, \fa^t)$ be a triple with $t>0$, let $I$ be an $\m$-primary ideal, $\fb \subseteq R$ be a proper ideal, and let $e$ be a positive integer such that $(p^e-1)(K_X+\D)$ is Cartier.
Then we define 
\[
\fjnn{n}{ u}(R, \D, \fa^t; \fb) : = \inf \{ s > 0 \mid \newtau{n}{u} (R, \D, \fa^t \fb^s) \subseteq I \} \in \R_{\ge 0}.
\]
\end{defn}

\begin{prop}\label{fpt basic}
With the above notation, the following hold.
\begin{enumerate}
\item $0 \le \fjnn{n}{u}(R, \D, \fa^t ; \fb) \le \lul + \mu_R(\fb)$.
\item $p^{en} \cdot \fjnn{n}{u}(R, \D, \fa^t ;\fb) \in \Z$.
\end{enumerate}
\end{prop}

\begin{proof}
By Proposition \ref{upper test basic} (5), we have
\begin{eqnarray*}
\newtau{n}{u}(R, \D, \fa^t \fb^{\lul+\mu_R(\fb)}) &=& \fb^{\lul} \cdot \newtau{n}{u} (R, \D, \fa^t \fb^{\mu_R(\fb)})\\
& \subseteq & \fb^{\lul} \subseteq I,
\end{eqnarray*}
which proves the assertion in (1).

The assertion in (2) follows from Proposition \ref{upper test basic} (1). 
\end{proof}

\begin{prop}\label{ACC for bdd}
Let $(X=\Spec R,\D)$ be a pair such that $(p^e-1)(K_X+\D)$ is Cartier for some positive integer $e$, let $t >0$ be a rational number, and let $M, \mu>0$ and $u \ge 2$ be positive integers.
Assume that 
\begin{enumerate}
\item $q> \mu+ \emb(R)$, and
\item $q^m(q-1) t \in \N$ for some integer $m$.
\end{enumerate}
Then, there exists a positive integer $n_1$ such that for every ideal $\fb \subseteq R$, 
if $\fb  = \fa + \m^M$ for some ideal $\fa \subseteq R$ with $\mu_R(\fa) \le \mu$, then we have $\newtau{n}{u}(R,\D, \fb^t)= \newtau{n_1}{u}(R, \D , \fb^t)$ for every $n \ge n_1$.
\end{prop}

\begin{proof}
By Proposition \ref{upper test basic} (6), it is enough to show the assertion in the case when $t > \mu+\emb(R)$ and $(p^e-1)t \in \N$.
Set $n_1 : = \len( \tau(R,\D) / (\m^{M \age{t}} \cdot \tau(R,\D)))$.
We will prove that the assertion holds for this constant $n_1$.

Let $\fa \subseteq R$ be an ideal such that $\mu_R(\fa) \le \mu$ and set $\fb : = \fa + \m^M$.
We consider the sequence of ideals $\{ \newtau{n}{u} (R, \D, \fb^t)\}_{n \ge 1}$.
As in the proof of Proposition \ref{upper test dc}, by using Lemma \ref{Skoda} (2) instead of Lemma \ref{Skoda} (1), the sequence $\{ \newtau{n}{u}(R, \D, \fb^t) \}_{n}$ is a descending chain.
Moreover, since $\fb \supseteq \m^M$, we have 
\begin{eqnarray*}
\newtau{n}{u}(R, \D, \fb^t) &\supseteq& \newtau{n}{u}(R, \D, (\m^{M})^{t})\\
&\supseteq & \newtau{n}{u}(R, \D, (\m^{M})^{t}) \\
& \supseteq & \newtau{n}{0}(R, \D, (\m^{M})^{ t}) \\
& \supseteq & \m^{M \age{t}} \cdot \tau(R, \D).
\end{eqnarray*}
Since we have
\[
\tau(R,\D) \supseteq \newtau{1}{u} (R, \D, \fb^t) \supseteq \newtau{2}{u} (R, \D, \fb^t) \supseteq \dots \supseteq \m^{M \age{t}} \cdot \tau(R, \D),
\]
there exists an integer $1 \le m \le n_1$ such that
\[
\newtau{m}{u}(R,\D, \fb^t)= \newtau{m+1}{u}(R, \D, \fb^t).
\]

On the other hand, as in the proof of Proposition \ref{upper test basic} (5), by using Lemma \ref{Skoda} (2) instead of Lemma \ref{Skoda} (1), we have
\[
\newtau{m+1}{u}(R, \D, \fb^{t'+1}) = \fb \cdot \newtau{m}{u}(R, \D, \fb^{t'})
\]
for any real number $t'>\mu+\emb(R)$.
Then, as in the proof of Proposition \ref{stab basic} (3), we have $\newtau{m+1}{u}(R, \D, \fa^t)= \newtau{m+2}{u}(R, \D, \fa^t)$, which completes the proof.
\end{proof}

%%%%%%%%%%%%%%%%%%%%%%%%%%%%%%%%%%%%%%%%%%%%%%%%%%%%%%%%%%%%%%%%%%%%%%
%%%%%%%%%%%%%%%%%%%%%%%%%%%%%%%%%%%%%%%%%%%%%%%%%%%%%%%%%%%%%%%%%%%%%%
\section{Rationality of the limit of $F$-pure thresholds}
%%%%%%%%%%%%%%%%%%%%%%%%%%%%%%%%%%%%%%%%%%%%%%%%%%%%%%%%%%%%%%%%%%%%%%
%%%%%%%%%%%%%%%%%%%%%%%%%%%%%%%%%%%%%%%%%%%%%%%%%%%%%%%%%%%%%%%%%%%%%%
In this section, we give uniform bounds for the denominators of $F$-jumping numbers (Proposition \ref{jump fin colen}) and for the stabilization exponents (Proposition \ref{stab fin colen}) of $\m$-primary ideals with fixed colength.
By using these bounds, we will prove Theorem \ref{intro BMS}.

\begin{prop}\label{jump fin colen}
Let $(X=\Spec R, \D)$ be a pair such that $(p^e-1)(K_X+\D)$ is Cartier for some integer $e>0$ and $M>0$ be an integer. 
Then there exists an integer $N>0$ such that for any ideal $\fa \subseteq R$, if $\fa \supseteq \m^M$, then any $F$-jumping number of $(R,\D; \fa)$ is contained in $(1/N)\cdot \Z$.
\end{prop}

\begin{proof}
Set $l : = \len(R/ \m^M) +\mu_R(\m^M)$ and $n : = \len( \tau(R, \D) / \tau(R, \D, \m^{M l}))$.
We note that the module $\tau(R, \D) / \tau(R, \D, \m^{M l})$ has finite length because the test ideals commute with localization (\cite[Proposition 3.1]{HT}).
Let $\fa \subseteq R$ be an ideal such that $\m^M \subseteq \fa$ and let $B \subseteq \R_{>0}$ be the set of all $F$-jumping numbers of $(R, \D ;\fa)$.

Since we have $\mu( \fa) \le l$, it follows from \cite[Corollary 3.27]{BSTZ} that for every element $b \in B \cap \R_{> l}$, we have $b-1 \in B$.
It also follows from \cite[Lemma 3.25]{BSTZ} that for every element $b \in B$, we have $p^e b \in B$.
Moreover, since $\tau(R, \D) \supseteq \tau(R, \D, \fa^t) \supseteq \tau(R, \D, \m^{M l})$ for every $t \le l$, the number of the set $B \cap [0, l]$ is at most $n$. 
Then the assertion follows from the lemma below.
\end{proof}

\begin{lem}
Let $l, n>0$ and $q \ge 2$ be integers.
Then there exists an integer $N>0$ with the following property:
if $B \subseteq \R_{\ge 0}$ is a subset such that
\begin{enumerate}
\item for every element $b \in B$, if $b >l$, then we have $b-1 \in B$,
\item if $b \in B$, then $q \cdot b \in B$, and
\item the number of the set $B \cap [0, l]$ is at most $n$,
\end{enumerate}
then we have $B \subseteq (1/N) \cdot \Z$.
\end{lem}

\begin{proof}
The proof is essentially the same as that of \cite[Proposition 3.8]{BMS1}.
Set $N : = q^n (q^{n !}-1)$, where $n!$ is the factorial of $n$.

For every element $b \in B$ and every integer $m \ge 0$, we define $b_m \in B \cap [0, l]$ by 
\[
b_m := (q^m b - \lfloor q^m b \rfloor)+ \min \{ l-1, \lfloor q^m b \rfloor \}.
\]

If $b \not\in (1/N) \cdot \Z$, then $b_0, b_1, \dots, b_n$ are all distinct and hence contradiction.
\end{proof}

\begin{prop}\label{stab fin colen}
Let $(X=\Spec R, \D)$ be a pair such that $(p^e-1)(K_X+\D)$ is Cartier for some integer $e>0$ and $M>0$ be an integer.
Then there exists $u_0>0$ such that for every ideals $\fa \supseteq \m^M$, we have
\[
\ustab( R, \D, \fa ;e) \le u_0.
\]
\end{prop}

\begin{proof}
Set $l := \len(R/ \m^M) + \mu_R (\m^M)$ and take an integer $n_0 > 0$ such that $p^{e (n_0 -1)}>l$.
Let $\fa \subseteq R$ be an ideal such that $\fa \supseteq \m^M$ and $t>0$ be a rational number such that $(p^e-1) t \in \N$.

We first consider the case when $l <t \le l p^{e n_0} $.
In this case, by Proposition \ref{lower test basic} (1), the sequence $\{ \tau^{en}_{+}(R, \D, \fa^t) \}_{n \ge 0}$ is an ascending chain such that  
\[
\tau(R, \D) \supseteq \tau^{en}_{+}(R, \D, \fa^t) \supseteq \tau^{0}_{+}(R, \D, \fa^t) = \fa^{\age{t}} \cdot \tau(R, \D) \supseteq \m^{l M p^{e n_0} } \cdot \tau(R, \D)
\]
for every $n$.
Therefore, there exists an integer $0 \le n < \len(\tau(R, \D)/ (\m^{l M p^{e n_0} } \cdot \tau(R, \D)))$ such that 
\[
\tau^{en}_{+}(R, \D, \fa^t)=\tau^{e(n+1)}_{+}(R, \D, \fa^t).
\]
By Proposition \ref{stab basic} (3), we have 
\[
\stab(R, \D, \fa^t ;e) \le n \le \len(\tau(R, \D)/ (\m^{l M p^{e n_0} } \cdot \tau(R, \D))).
\]

We next consider the case when $t \le l$.
Since $l <t p^{e n_0}  \le l p^{e n_0}  $, it follows from Proposition \ref{stab basic} (2) that
\begin{eqnarray*}
\stab(R, \D, \fa^t ; e) &\le& \stab(R, \D, \fa^{t p^{e n_0}} ; e) +n_0\\
 &\le& \len(\tau(R, \D)/ (\m^{l M p^{e n_0} } \cdot \tau(R, \D))) +n_0.
\end{eqnarray*}
Therefore, $u_0 : =\len(\tau(R, \D)/ (\m^{l M p^{e n_0}} \cdot \tau(R, \D))) +n_0$ satisfies the property.
\end{proof}

\begin{prop}\label{cata tau}
Let $(X=\Spec R, \D)$ be a pair such that $(p^e-1)(K_X+\D)$ is Cartier for some integer $e>0$, $\{ \fa_m \}_{m \in \N}$ be a family of ideals of $R$ and $t>0$ be a real number.
Fix a non-principal ultrafilter $\U$.
Let $(\cata{R}, \cata{\m})$ be the catapower of the local ring $(R, \m)$, $\cata{\D}$ be the flat pullback of $\D$ to $\Spec \cata{R}$ and $\fa_\infty := \catae{\fa_m} \subseteq \cata{R}$.
If there exists a positive integer $M$ such that $\fa_m \supseteq \m^M$ for every $m$, then we have
\[
\tau (\cata{R}, \cata{\D}, \fa_\infty^t) = \catae{ \tau(R, \D, \fa_m^t)} \subseteq \cata{R}.
\]
\end{prop}

\begin{proof}
We first consider the case when $t$ is a rational number.
By enlarging $e$, we may assume that $p^{en}(p^e-1)t \in \Z$ for some integer $n \ge 0$. 
Take a positive integer $u$ as in Proposition \ref{stab fin colen}.
Then we have
\[
\tau(R,\D, \fa_m^t) = \tau^{e(n+u)}_{+}(R, \D, \fa^t_m),
\]
for every $m$.
By enlarging $u$, we may assume that
\[
\tau(\cata{R}, \cata{\D}, \fa_\infty^t) = \tau^{e(n+u)}_{+}(\cata{R}, \cata{\D}, \fa_\infty^t).
\]
Since $\mu_R(\fa_m) \le \len(R/\m^M) + \mu_R(\m^M)$ for every $m$, it follows from Lemma \ref{ulim prod} that 
\[
(\fa_\infty)^s=\catae{(\fa_m)^s}
\]
for every integer $s>0$.
Combining with Proposition \ref{test base change} and \ref{ultra field ext}, we have 
\begin{eqnarray*}
\tau^{e l}_{+} (\cata{R}, \cata{\D}, \fa_\infty^t) &=&\phi^{e l}_{\cata{\D}}(F^{e l}_*(\fa^{\age{t p^{e l}}}_{\infty} \cdot \tau(\cata{R}, \cata{\D})))\\
&=& \phi^{e l}_{\cata{\D}} (F^{e l}_* \catae{\fa^{\age{t p^{e l }}}_{m} \cdot \tau(R,\D)})\\
&=& \catae{\phi^{e l}_{\D}(F^{e l}_*(\fa^{\age{t p^{e l}}}_{m} \cdot \tau(R,\D)))}\\
&=& \catae{ \tau^{e l}_{+}(R, \D, \fa_m^t)} \subseteq \cata{R}
\end{eqnarray*}
for every integer $l$.
Therefore, we have 
\[
\tau (\cata{R}, \cata{\D}, \fa_\infty^t) = \catae{ \tau(R, \D, \fa_m^t)} \subseteq \cata{R}.
\]

We next consider the case when $t$ is not a rational number.
For sufficiently large integer $n$, we have
\begin{eqnarray*}
\tau (\cata{R}, \cata{\D}, \fa_\infty^t) &=& \tau^{en}_{+}(\cata{R},\cata{\D}, \fa_\infty^t)\\
& =& \catae{\tau^{en}_{+}(R, \D, \fa_m^t)} \\
&\subseteq& \catae{ \tau(R, \D, \fa_m^t)} \subseteq \cata{R}.
\end{eqnarray*}
For the converse inclusion, by Proposition \ref{disc rat}, we can take a rational number $t'$ such that $t' < t$ and $\tau(\cata{R}, \cata{\D}, \fa_\infty^{t})=\tau(\cata{R}, \cata{\D}, \fa_\infty^{t'})$.
Then, we have
\begin{eqnarray*}
\tau(\cata{R}, \cata{\D}, \fa_\infty^{t})&=&\tau(\cata{R}, \cata{\D}, \fa_\infty^{t'})\\
&=& \catae{\tau(R, \D, \fa_m^{t'})} \\
& \supseteq & \catae{\tau(R, \D, \fa_m^{t})},
\end{eqnarray*}
which completes the proof.
\end{proof}

\begin{prop}\label{fpt sh}
With the notation above, let $I \subseteq R$ be an $\m$-primary ideal.
Assume that $\m^M \subseteq \fa_m \subseteq \m$ for every $m$.
Then there exists $T \in \U$ such that for all $m \in T$, we have
\[
\fjn^I (R, \D; \fa_m)= \fjn^{I \cdot R_{\#}} ( \cata{R}, \cata{\D}, \fa_\infty).
\]
\end{prop}

\begin{proof}
Set $t := \fjn^{I \cdot \cata{R}}(\cata{R}, \cata{\D} ; \fa_\infty) \in \R_{\ge 0}$.
If $\tau(R, \D) \subseteq I$, then we have $\fjn^I (R, \D ; \fa_m)=0$ for every $m \in \N$ and $\fjn^{I \cdot \cata{R}} (\cata{R}, \cata{\D}, \fa_\infty ) =0$.
Therefore, we may assume that $\tau(R, \D) \not\subseteq I$.
Since $\fa_\infty \neq (0)$, it follows from Lemma \ref{test inc} (2) that $t >0$.

It follows from Proposition \ref{cata tau} that we have 
\[
\catae{ \tau( R, \D, \fa_m^t)} = \tau( \cata{R}, \cata{\D}, \fa_\infty^t) \subseteq I \cdot \cata{R}.
\]
Since $I$ is $\m$-primary, it follows from Lemma \ref{ultra incl} that there exists $S_1 \in \U$ such that $\tau(R, \D, \fa_m^t) \subseteq I$ for every $m \in S_1$.
Therefore $\fjn^I (R, \D; \fa_m) \le \fjn^{I \cdot \cata{R}} (\cata{R}, \cata{\D}, \fa_\infty)$ for every $m \in S_1$.

On the other hand, by Proposition \ref{jump fin colen}, there exists $0 < t' < t$ such that for every ideal $\fb \supseteq \m^M$, if $ t' < \fjn^I (R, \D; \fb)$, then $t \le \fjn^I (R, \D; \fb)$.
Since $t'<t$, we have
\[
\catae{ \tau( R, \D, \fa_m^{t'})} = \tau( \cata{R}, \cata{\D}, \fa_\infty^{t'}) \not\subseteq I \cdot \cata{R}.
\]
Hence, we have
\[
\ulim_{m} \tau(R, \D, \fa_m^{t'}) \not\subseteq \ultra{I}.
\]
Therefore, there exists $S_2 \in \U$ such that $\tau(R, \D, \fa_m^{t'}) \not\subseteq I$ for every $m \in S_2$.
Then $T := S_1 \cap S_2$ satisfies the assertion.
\end{proof}

\begin{lem}[\textup{\cite[Lemma 3.3]{BMS2}}]\label{subadd}
Let $(X=\Spec R, \D)$ be a pair such that $K_X+\D$ is $\Q$-Carter, $I$ be an $\m$-primary ideal, $\fa, \fb \subseteq R$ be proper ideals.
Then we have 
\[
\fjn^I(R, \D ;\fa+\fb) \le \fjn^I(R, \D; \fa)+ \fjn^I (R, \D; \fb).
\]
\end{lem}

\begin{proof}
As in the proof of \cite[Theorem 3.1]{Tak04}, for every real number $c \ge 0$, we can show that
\[
\tau(R, \D, (\fa+\fb)^c) = \sum_{u,v \ge 0, u+v =c} \tau( R, \D, \fa^u \fb^v).
\]

Set $t := \fjn^I(R, \D; \fa)$ and $s := \fjn^I(R, \D; \fb)$.
Then we have
\[
\tau(R, \D, (\fa+\fb)^{t+s}) = \sum_{u,v \ge 0, u+v =s+t} \tau( R, \D, \fa^u \fb^v) \subseteq \tau (R, \D, \fa^t) + \tau(R, \D, \fb^s) \subseteq I.
\]
\end{proof}

\begin{thm}[\textup{Theorem \ref{intro BMS}, cf. \cite[Theorem 1.2]{BMS2}}]\label{fpt sh2}
Let $(X=\Spec R, \D)$ be a pair such that $(p^e-1)(K_X+\D)$ is Cartier for some integer $e>0$, $(\cata{R}, \cata{\m})$ be the catapower of $(R, \m)$, $\cata{\D}$ be the flat pullback of $\D$ to $\Spec \cata{R}$, $I \subseteq R$ be an $\m$-primary ideal, $\{ \fa_m \}_{m \in \N}$ be a family of proper ideals and $\fa_\infty : = \catae{\fa_m} \subseteq \cata{R}$.
Then we have 
\[
\sh( \ulim_m \fjn^I (R, \D; \fa_m) ) = \fjn^{I\cdot R_\#} (\cata{R}, \cata{\D}, \fa_\infty) \in \Q.
\]
In particular, if the limit $\lim_{m \to \infty} \fjn^I (R, \D; \fa_m)$ exists, then we have
\[
\lim_{m\to \infty} \fjn^I (R, \D; \fa_m) = \fjn^{I\cdot R_\#} (\cata{R}, \cata{\D}, \fa_\infty).
\]
\end{thm}

\begin{proof}
The proof is essentially the same as the proof of \cite[Theorem 1.2]{BMS2}.
If $\tau(R, \D) \subseteq I$, then the assertion in the theorem is trivial.
Therefore, we may assume that $\tau(R, \D) \not\subseteq I$.

For every integer $M>0$, we set $\fb_{\infty ,M} : = \fa_\infty + (\cata{\m})^M$ and $\fb_{m,M} : = \fa_m + \m^M$ for every integer $m$.
We write $s := \fjn^{I \cdot \cata{R}}(\cata{R}, \cata{\D}; \cata{\m})$

By Lemma \ref{subadd}, we have 
\begin{equation}\label{BMS1}
| \fjn^{I \cdot \cata{R}} (\cata{R}, \cata{\D} ; \fa_\infty) - \fjn^{I \cdot \cata{R}} (\cata{R}, \cata{\D}; \fb_{\infty, M}) | \le s/M
\end{equation}
for every $M$.

By Proposition \ref{test base change} (4), we have $s = \fjn^I(R, \D; \m)$.
Therefore, it follows from Lemma \ref{subadd} that
\begin{equation}\label{BMS3}
| \fjn^I(R, \D; \fa_m) - \fjn^I (R, \D; \fb_{m,M}) | \le s/M
\end{equation}
for every $m$ and $M$.

On the other hand, since $\fb_{\infty, M} = \catae{\fb_{m, M}}$, it follows from Proposition \ref{fpt sh} that there exists $T_M \in \U$ such that 
\begin{equation}\label{BMS2}
\fjn^{I \cdot \cata{R}} (\cata{R}, \cata{\D}; \fb_{\infty, M}) = \fjn^I(R, \D; \fb_{m,M})
\end{equation}
for every $m \in T_M$.

By combining the equations (\ref{BMS1}), (\ref{BMS3}), and (\ref{BMS2}), we have
\[
| \fjn^{I \cdot \cata{R}} (\cata{R}, \cata{\D}; \fa_\infty) - \fjn^I(R, \D; \fa_{m})| \le 2s/M
\]
for every $m \in T_M$.

It follows from the definition of the shadow that 
\[
\sh( \ulim_m \fjn^I (R, \D; \fa_m) ) = \fjn^{I \cdot \cata{R}} (\cata{R}, \cata{\D} ; \fa_\infty),
\]
which completes the proof.
\end{proof}

%%%%%%%%%%%%%%%%%%%%%%%%%%%%%%%%%%%%%%%%%%%%%%%%%%%%%%%%%%%%%%%%%%%%%%
%%%%%%%%%%%%%%%%%%%%%%%%%%%%%%%%%%%%%%%%%%%%%%%%%%%%%%%%%%%%%%%%%%%%%%
\section{Proof of Main Theorem}
%%%%%%%%%%%%%%%%%%%%%%%%%%%%%%%%%%%%%%%%%%%%%%%%%%%%%%%%%%%%%%%%%%%%%%
%%%%%%%%%%%%%%%%%%%%%%%%%%%%%%%%%%%%%%%%%%%%%%%%%%%%%%%%%%%%%%%%%%%%%%
In this section, we introduce Condition $(\star)$ (Definition \ref{def condA}) which plays the key role in the proof of the main theorem and we prove some properties of Condition $(\star)$ (Proposition \ref{B to A} and Proposition \ref{A to C}).
By combining them with Proposition \ref{ACC for bdd} and Theorem \ref{fpt sh2}, we give the proof of the main theorem (Theorem \ref{main}).

\begin{obs}\label{obs}
Let $X$ be a normal variety over a field $k$ of characteristic zero, $\D$ be an effective $\Q$-Weil divisor on $X$ such that $K_X+\D$ is $\Q$-Cartier, $\fa \subseteq \sO_X$ be a non-zero coherent ideal sheaf, $t \ge 0$ be a rational number, $x \in X$ be a closed point and $\m_x \subseteq \sO_X$ be the maximal ideal at $x$.
We consider the \emph{log canonical threshold} 
\[
\lct_x(X, \D, \fa^t ; \m) : = \inf \{ s \ge 0 \mid (X, \D, \fa^t \m^{s}) \textup{ is not log canonical at }x \}.
\]

By considering a log resolution of $(X, \D)$, $\fa$ and $\m$, we can show that there exist a real number $t'<t$ and rational numbers $a, b$ such that
\begin{equation}\label{LC polytope}
\lct_x (X, \D, \fa^s ; \m) = as +b
\end{equation}
for every $t' < s < t$.

Assume that there exist integers $q \ge 2$ and $m \ge 0$ such that $q^m(q-1) t \in \N$.
Then for every $n>m$, the $n$-th digit of $t$ in base $q$ satisfies $t^{(n)}=l$ for some constant $l>0$.
Set $N := -a l /q$.
Then we have
\begin{equation}\label{condA char0}
\lct_x(X, \D, \fa^{\qadic{t}{n+1}}; \m) = \lct_x(X, \D, \fa^{\qadic{t}{n}} ; \m) -N/q^n
\end{equation}
for sufficiently large $n$.
\end{obs}

Motivated by the observation above, we define the following condition.

\begin{defn}\label{def condA}
Let $(X=\Spec R, \D, \fa^t)$ be a triple such that $t>0$ and $(p^e-1)(K_X+\D)$ is Cartier for some integer $e>0$, $I \subseteq R$ be an $\m$-primary ideal and $u, N \ge 0$ be integers.
We say that $(R, \D, \fa^t, I, e, u, N)$ satisfies \emph{Condition $(\star)$} if for every $n \ge 0$, we have 
\[
\fjnn{n+1}{u}(R,\D, \fa^t; \m) \ge \fjnn{n}{u}(R, \D, \fa^t; \m) -N / p^{en}.
\]

\end{defn}

\begin{rem}\label{CondA rmk}
If we have $u \ge \ustab(R, \D, \fa, \m ;e)$, then we have 
\[
\fjnn{n}{u}(R, \D, \fa^t ; \m) = \langle \fjn^I (R, \D, \fa^{\qadic{t}{n}} ; \m) \rangle^{n,q},
\]
where we write $q : = p^e$.
Therefore, Condition $(\star)$ can be regarded as an analogue of the equation \ref{condA char0} in Observation \ref{obs}.
See also Corollary \ref{CondA single} below.

We also note that the equation \ref{LC polytope} in Observation \ref{obs} may not hold for $F$-pure thresholds (cf. \cite[Example 5.3]{Per}).
\end{rem}

We first give a sufficient condition for Condition $(\star)$.

\begin{prop}\label{B to A}
Let $(X=\Spec R, \D, \fa^t)$ be a triple such that $t>0$ and $(p^e-1)(K_X+\D)$ is Cartier for some $e>0$, let $I \subseteq R$ be an $\m$-primary ideal, let $0 < l< p^e$ be a positive integer and let $n_0 \ge 0 $ and $ u \ge 2$ be integers.
Set 
\[
q=p^e, \ N := q^{n_0+3} \emb(R), \ t_0 : =\frac{q^2}{q-1},  \textup{ and } M_0 := \frac{(q^{n_0+6}-1) \emb(R) }{q-1}.
\]
Assume that 
\begin{enumerate}
\item $q > \mu_R(\fa)$, 
\item $q>\lul$,
\item the $n$-th digit of $t$ in base $q$ satisfies $t^{(n)}=l$ for every $n \ge 2$, and 
\item $\newtau{n_0+1}{ u}(R, \D, \fa^{l t_0}) + \m^{M_0}  \cdot \tau(R, \D) \supseteq \newtau{n_0}{ u} (R, \D, \fa^{l t_0})$.
\end{enumerate}
Then, $(R, \D, \fa^t, I,  e, u, N)$ satisfies Condition $(\star)$
\end{prop}

\begin{proof}
By induction on $n \ge 0$, we will show the inequality 
\begin{eqnarray}\label{eqn A}
\fjnn{n+1}{u}(R,\D, \fa^t; \m) \ge \fjnn{n}{u}(R, \D, \fa^t; \m) -N / q^n.
\end{eqnarray}

{\bf Step.1}
We consider the case when $n \le n_0 +2$.
In this case, we have 
\[
N/q^n \ge q \cdot \emb(R) \ge \lul+ \emb(R) .
\]
By Proposition \ref{fpt basic} (1), we have
\[
\fjnn{n}{u}(R, \D, \fa^t ; \m) \le \lul + \emb(R).
\]
Hence we have 
\[
\fjnn{n}{u}(R, \D, \fa^t; \m) -N / q^n \le 0,
\] 
which implies the inequality \ref{eqn A}.

{\bf Step.2}
From now on, we assume $n \ge n_0 + 3$.
Set $r : = q^n \cdot \fjnn{n}{u}(R, \D, \fa^t; \m) $.
By Proposition \ref{fpt basic}, we have $r \in \Z$.
We first consider the case when 
\[
r \le q^{n_0} \cdot \emb(R).\]
In this case, we have
\[
 \fjnn{n}{u}(R, \D, \fa^t; \m) -N / q^n \le 0,
\] 
which shows the inequality \ref{eqn A}.
Therefore, we may assume $r> q^{n_0} \cdot  \emb(R).$

{\bf Step.3} 
Set $s: = \age{r/q^{n_0}} -\emb(R)-1 $ and $s' : = \age{(s+M_0)/q^2}$.

In this step, we will show the inclusion
\begin{equation}\label{B to A Step3}
\newtau{n}{u} (R ,\D, \fa^t \m^{r/q^n}) \subseteq \newtau{n+1}{u} (R, \D, \fa^t \m^{s/q^{n-n_0}} ) + \newtau{n-n_0-2}{2}(R, \D, \fa^t \m^{s'/q^{n-n_0+2}}).
\end{equation}

By the assumption (3), $\alpha:= tq^{n-n_0} - l t_0 = q^2 \age{tq^{n-n_0-2} -1}$ is an integer.
It follows from Proposition \ref{upper test basic} (1), (5), and (6) that 
\begin{eqnarray*}
\newtau{n}{u}(R,\D,\fa^t \m^{r/q^n}) 
&=&\phi^{e(n-n_0)}_\D(F^{e(n-n_0)} (\newtau{n_0}{u}(R,\D, \fa^{tq^{n-n_0}} \m^{r/q^{n_0}})))\\
&\subseteq & \phi^{e(n-n_0)}_\D(F^{e(n-n_0)} (\fa^{\alpha} \m^{s} \newtau{n_0}{u}(R,\D, \fa^{l t_0}))).
\end{eqnarray*}

Similarly, we have
\begin{eqnarray*}
\newtau{n+1}{u}(R,\D,\fa^t \m^{s/q^{n-n_0}}) 
&=&\phi^{e(n-n_0)}_\D(F^{e(n-n_0)} (\newtau{n_0+1}{u}(R,\D, \fa^{tq^{n-n_0}} \m^{s})))\\
&\supseteq & \phi^{e(n-n_0)}_\D(F^{e(n-n_0)} (\fa^{\alpha} \m^{s} \newtau{n_0+1}{u}(R,\D, \fa^{l t_0} ))).
\end{eqnarray*}

On the other hand, it follows from the definitions that
\begin{eqnarray*}
\newtau{n-n_0-2}{2}(R, \D, \fa^t \m^{s'/q^{n-n_0-2}})
&=& \phi^{e(n-n_0)}_\D(F^{e(n-n_0)} (\fa^{\alpha} \m^{q^2(s'-1)} \tau(R,\D)))\\
&\supseteq& \phi^{e(n-n_0)}_\D(F^{e(n-n_0)} (\fa^{\alpha} \m^{s+M_0} \tau(R,\D))).
\end{eqnarray*}

By combining them with the assumption (4), we have the inclusion \ref{B to A Step3}.

{\bf Step.4}
In this step, we will show the inclusion
\begin{equation}\label{B to A Step4}
\newtau{ n-n_0-2}{ 2 }(R, \D, \fa^t \m^{s'/q^{n-n_0-2}}) \subseteq I.
\end{equation}

It follows from the induction hypothesis that 
\[
\fjnn{n}{u}(R, \D, \fa^t; \m) \ge \fjnn{n-n_0-2}{ u}(R,\D, \fa^t; \m) - (\sum_{i=n-n_0-2}^{n-1} \frac{N}{q^i}).
\]
Therefore, we have the inequality
\begin{eqnarray*}
&&\frac{s'}{q^{n-n_0-2}} \ge \frac{s+M_0}{q^{n-n_0}}  \ge  \frac{ r/q^{n_0} -\emb(R) -1+ M_0}{q^{n-n_0}} \\
&= & \fjnn{n}{u}(R, \D, \fa^t; \m) + \frac{ -\emb(R)-1 + M_0}{q^{n-n_0}} \\
& \ge & \fjnn{n-n_0-2}{u}(R,\D, \fa^t; \m) - (\sum_{i=n-n_0-2}^{n-1} \frac{N}{q^i}) + \frac{-\emb(R) -1 + M_0}{q^{n-n_0}} \\
& > & \fjnn{n-n_0-2}{ u}(R,\D, \fa^t; \m).
\end{eqnarray*}

Since we have $u \ge 2 $, It follows from Proposition \ref{upper test basic} (7) that
\[
\newtau{n-n_0-2}{ 2}(R, \D, \fa^t \m^{s'/q^{n-n_0-2}}) \subseteq \newtau{n-n_0-2}{  u}(R, \D, \fa^t \m^{s'/q^{n-n_0-2}}) \subseteq I.
\]

{\bf Step.5}
It follows from Proposition \ref{upper test basic} (1) that 
\[
\newtau{ n}{ u} (R ,\D, \fa^t \m^{r/q^n}) \not\subseteq I.
\]
Combining it with the inclusions \ref{B to A Step3} and \ref{B to A Step4}, we have
\[
\newtau{ n+1}{ u}(R, \D, \fa^t \m^{s/q^{n-n_0}} ) \not\subseteq I.
\]
Hence, we have
\begin{eqnarray*}
\fjnn{n+1}{u}(R, \D, \fa^t; \m) & \ge & \frac{s}{q^{n-n_0}}\\
& \ge & \frac{r/q^{n_0} -\emb(R)-1}{q^{n-n_0}}\\
&=& \fjnn{n}{u}(R, \D, \fa^t; \m) - \frac{\emb(R)+1}{q^{n-n_0}}\\
&>& \fjnn{n}{u}(R, \D, \fa^t; \m) -\frac{N}{q^n},
\end{eqnarray*}
which completes the proof of the proposition.
\end{proof}

\begin{cor}\label{CondA single}
Let $(X=\Spec R, \D, \fa^t)$ be a triple such that $t>0$ is a rational number and $(p^e-1)(K_X+\D)$ is Cartier for some integer $e>0$ and $I \subseteq R$ be an $\m$-primary ideal.
Then, there exist integers $e', u_0, N >0$ such that for every $u \ge u_0$, $(R,\D, \fa^t, I, e', u, N)$ satisfies Condition $(\star)$.
In particular, there exists an integer $N'>0$ such that if we write $q : = p^{e'}$, then
\[
\fjn^I (R, \D, \fa^{\qadic[q]{t}{n+1}} ;\m ) \ge \fjn^I(R,\D,\fa^{\qadic[q]{t}{n}}; \m) - N'/{q}^n
\]
for every integer $n \ge 0$.
\end{cor}

\begin{proof}
Take an integer $m>0$ such that $q : = p^{em}$ satisfies the assumptions (1), (2), and (3) in Proposition \ref{B to A}.

Set $l = t^{(2)}$ and $t_0 : = q^2/(q-1)$.
Then it follows from Proposition \ref{disc rat} that there exists an integer $n_0>0$ such that 
\[
\tau(R, \D, \fa^{\qadic{l t_0}{n_0}}) = \tau(R, \D, \fa^{\qadic{l t_0}{(n_0 +1)}}).
\]

Set $e' : =em$, $u_0 : = \ustab(R, \D, \fa ; e')$ and $N :=q^{n_0+3} \cdot \emb(R)$.
Then the first assertion follows from Proposition \ref{B to A}.

Set $N' : =N+1$.
Then the second assertion follows from Remark \ref{CondA rmk}. 
\end{proof}

\begin{prop}\label{A to C}
Suppose that $(R,\D,\fa^t)$, $q=p^e$, $u$, and $N$ satisfies the conditions of Proposition \ref{B to A}.
We further assume that $q>\lul+\mu_R(\fa) +\emb(R)$.
Then for every $n \ge 1$, we have 
\[
\fjnn{n}{u}(R, \D, \fa^t; \m)= \fjnn{n}{u}(R, \D, \fb^t;\m),
\]
where $\fb : = \fa + \m^{q^{u+2} \cdot N}$.
In particular, for every $n$, we have
\[
\newtau{n}{u}(R, \D, \fa^t) \subseteq I \textup{ if and only if } \newtau{n}{u}(R, \D, \fb^t) \subseteq I.
\]
\end{prop}

\begin{proof}
Set $M := q^{u+2} \cdot N$, $M' :=q^{u+1} \cdot N$, $s_n : =\fjnn{n}{u}(R, \D, \fa^t; \m)$, and $\delta_n:=q^n s_n $ for every integer $n$.
By Proposition \ref{fpt basic} (2), we have $\delta_n \in \N$.
It is enough to show the following claim.
\begin{cl}
For every $n \ge 1$ and every ideal $\q \subseteq \m^{\max\{0, q^u \cdot \delta_n -M'\}} \cdot \tau(R, \D)$, we have 
\[
\tau^{n,u}_{e,\q}(R, \D, \fa^t) \equiv \tau^{n,u}_{e,\q}(R,\D, \fb^t) \ (\mod{I}).
\]
\end{cl}

In fact, if the claim holds, then it follows from Proposition \ref{upper test basic} (4) that
\begin{eqnarray*}
\newtau{n}{u} (R, \D, \fb^t \m^{s_n + \epsilon} ) 
&\equiv& \newtau{n}{u} (R, \D, \fa^t \m^{s_n+\epsilon } ) \ (\mod{I})\\
& \subseteq & I
\end{eqnarray*}
for every real number $0 < \epsilon \le 1/q^n$.
Therefore we have 
\[
\fjnn{n}{u}(R, \D, \fa^t; \m) \ge \fjnn{n}{u}(R, \D, \fb^t; \m).
\]

Similarly, if $s_n>0$, then we have 
\begin{eqnarray*}
\newtau{n}{u} (R, \D, \fb^t \m^{s_n} ) 
&\equiv& \newtau{n}{u}(R, \D, \fa^t \m^{s_n} ) \ (\mod{I})\\
& \not\subseteq& I,
\end{eqnarray*}
which shows $\fjnn{n}{u}(R, \D, \fa^t; \m) \le \fjnn{n}{u}(R, \D, \fb^t; \m)$.
Since this inequality also holds when $s_n=0$, we complete the proof of the proposition.

\end{proof}

\begin{clproof}
We use induction on $n$.

{\bf Step.1} 
We first consider the case when $n=1$.
It follows from Proposition \ref{upper test basic} (3) that 
\[
\tau^{n,u}_{e,\q}(R, \D, \fa^t) \equiv \tau^{n,u}_{e,\q}(R,\D, \fb^t) \ (\mod{\tau^{n,u}_{e,\q \cdot \m^M} }(R,\D) ).
\]

Since we have $\q \cdot \m^M \subseteq \m^{q^u \age{q(\lul+\emb(R))-1}} \cdot \tau(R,\D)$, it follows from Proposition \ref{upper test basic} (2), (4) and (5) that
\[
\tau^{n,u}_{e,\q \cdot \m^M} (R,\D) \subseteq \m^{\lul} \subseteq I.
\]
Therefore, the assertion holds when $n=1$.

{\bf Step.2}
From now on, we consider the case when $n \ge 2$.
Set $\q' := \phi_\D^e(F^e_*(\fa^{t^{(n)} \cdot q^u} \q))$ and $\q'' := \phi_\D^e(F^e_*(\fb^{t^{(n)} \cdot q^u} \q))$. 

Then it follows from Proposition \ref{upper test basic} (9) that 
\begin{equation}\label{A to C Step2-1}
\tau^{n,u}_{e,\q}(R,\D,\fa^t) = \tau^{n-1,u}_{e,\q'} (R,\D, \fa^t).
\end{equation}

Similarly, by using Lemma \ref{Skoda} (2) instead of (1), we have
\begin{equation}\label{A to C Step2-2}
\tau^{n,u}_{e,\q}(R,\D,\fb^t) = \tau^{n-1,u}_{e,\q''} (R,\D, \fb^t).
\end{equation}

{\bf Step.3}
In this step, we will show the equation
\begin{equation}\label{A to C Step3}
\tau^{n-1,u}_{e,\q'} (R,\D, \fa^t) \equiv \tau^{n-1,u}_{e,\q''} (R,\D, \fa^t) \ (\mod{I}).
\end{equation}

Set $J := \phi_\D^e(F^e_*(\m^M \q))$, then we have $\q' \equiv \q'' (\mod{ J })$.
By Proposition \ref{upper test basic} (3), it is enough to show that 
\[
\tau^{n-1,u}_{e,J}(R,\D, \fa^t) \subseteq I.
\]

Since we have $\delta_n \ge q \delta_{n-1}-q N$, it follows from Lemma \ref{Skoda} that
\begin{eqnarray*}
J &\subseteq& \phi^e_\D (\m^{q^u \delta_n +M - M' } \cdot \tau(R, \D)) \\
& \subseteq & \m^{(q^u \delta_n + M-M' )/q -\emb(R) } \cdot \tau(R, \D) \\
& \subseteq & \m^{q^u \delta_{n-1}} \cdot \tau(R, \D).
\end{eqnarray*}

Therefore, it follows from Proposition \ref{upper test basic} (2) and (4) that
\begin{eqnarray*}
\tau^{n-1,u}_{e,J}(R,\D, \fa^t) \subseteq \newtau{n-1}{u}(R, \D, \fa^t \m^{s_{n-1} +(1/q^{n-1})}) \subseteq I,
\end{eqnarray*}
which shows the equation \ref{A to C Step3}.

{\bf Step.4}
In this step, we will show the equation
\begin{equation}\label{A to C Step4}
\tau^{n-1,u}_{e,\q''} (R,\D, \fa^t) \equiv \tau^{n-1,u}_{e,\q''} (R,\D, \fb^t) \ (\mod{I}).
\end{equation}

As in Step 3, we have
\begin{eqnarray*}
\q'' & \subseteq & \phi^e_\D (F^e_*( \m^{\max \{ 0, q^u \delta_n -M' \}} \cdot \tau(R, \D)))\\
& \subseteq & \m^{\max\{0, q^{u-1} \delta_n -(M'/q)- \emb(R) \}} \cdot \tau(R, \D) \\
& \subseteq & \m^{\max\{0, q^u \delta_{n-1} -M' \}} \cdot \tau(R, \D).
\end{eqnarray*}
By induction hypothesis, we get the equation \ref{A to C Step4}.

By combining the equations \ref{A to C Step2-1}, \ref{A to C Step2-2}, \ref{A to C Step3} and \ref{A to C Step4}, we complete the proof of the claim.
\end{clproof}

\begin{cor}[Theorem \ref{intro key}]\label{key}
Let $(X=\Spec R, \D)$ be a pair such that $(p^e-1)(K_X+\D)$ is Cartier for some integer $e>0$, $I \subseteq R$ be an $\m$-primary ideal, $l, n_0 \ge 0$ and $u \ge 2$ be integers and $t>0$ be a rational number such that $p^e(p^e-1)t \in \N$.
We set $l : = t^{(2)}$, $t_0 : =p^{2e}/(p^e-1)$ and $M_0=(p^{e(n_0+6)}-1) \cdot \emb(R) / (p^e-1)$.
Then there exists an integer $n_1>0$ with the following property:
for any ideal $\fa \subseteq R$ such that 
\begin{enumerate}
\item $p^e>\mu_R(\fa) + \lul+\emb(R)$, and
\item $\newtau{n_0+1}{ u}(R, \D, \fa^{l t_0}) + \m^{M_0}  \cdot \tau(R, \D) \supseteq \newtau{n_0}{ u} (R, \D, \fa^{l t_0})$
\end{enumerate}
we have
\[
\newtau{n}{u}(R,\D, \fa^t) \subseteq I \textup{ if and only if } \newtau{n_1}{u}(R,\D, \fa^t) \subseteq I
\]
for every integer $n \ge n_1$.

%With the notation above, there exists an integer $n_1>0$ which depends only on $(R,\D)$, $t$, $I$, $e$, $u$, and $n_0$ such that 
%\[
%\newtau{n}{u}(R, \D, \fa^t) \subseteq I \textup{ if and only if } \newtau{n_1}{u}(R, \D, \fa^t) \subseteq I
%\]
%for every integer $n \ge n_1$.
\end{cor}

\begin{proof}
By Proposition \ref{B to A} and Proposition \ref{A to C}, $\fb : = \fa + \m^{q^{u+n_0+5} \emb(R)}$ satisfies
\[
\newtau{n}{u}(R,\D,\fa^t) \subseteq I \textup{ if and only if } \newtau{n}{u}(R,\D,\fb^t) \subseteq I.
\] 
for every integer $n$.

On the other hand, it follows from Proposition \ref{ACC for bdd} that there exists an integer $n_1>0$ which depends only on $\mu : = q- \emb(R)-1 $, $M:=q^{u+n_0+5} \emb(R)$, $e, u$, and $t$ such that for every integer $n>n_1$, we have
\[
\newtau{n}{u}(R,\D,\fb^t) \subseteq I \textup{ if and only if } \newtau{n_1}{u}(R,\D,\fb^t) \subseteq I,
\]
which completes the proof.
\end{proof}

By using the method of ultraproduct, we can apply Corollary \ref{key} to infinitely many ideals simultaneously.

\begin{prop}\label{CondA holds}
Let $(X=\Spec R, \D)$ be a pair such that $(p^e-1)(K_X+\D)$ is Cartier for some integer $e>0$, $I \subseteq R$ be an $\m$-primary ideal, $\{ \fa_m \}_{m \in \N}$ be a family of ideals of $R$, $t >0$ be a rational number, and $\U$ be a non-principal ultrafilter.
Assume that
\begin{enumerate}
\item $\tau(R, \D)$ is $\m$-primary or trivial,
\item $p^e > \mu_R(\fa_m) + \lul + \emb(R)$ for every $m$, and
\item $p^e(p^e-1) t \in \N$.
\end{enumerate}
Then for any sufficiently large integer $u > 0$, there exist an integer $n_1$ and $T \in \U$ such that 
\[
\newtau{n}{u}(R, \D, \fa_m^t) \subseteq I \textup{ if and only if } \newtau{n_1}{u}(R, \D, \fa_m^t) \subseteq I
\]
for every integer $n \ge n_1$ and $m \in T$.
\end{prop}

\begin{proof}
Set $t_0 : =p^{2e}/(p^e-1)$.
Since $p^e(p^e-1)t \in \N$, there exists an integer $0<l<p^e$ such that $t^{(n)}=l$ for every $n \ge 2$.
By Corollary \ref{key}, it is enough to show that for any sufficiently large integer $u>0$, there exist an integer $n_0$ and $T \in \U$ such that for every $m \in T$, we have 
\[
\newtau{n_0+1}{ u }(R, \D, \fa^{l t_0}) + \m^{M_0} \cdot \tau(R, \D) \supseteq \newtau{n_0}{ u} (R, \D, \fa^{l t_0}),
\]
 where $M_0 : = (p^{e(n_0+6)}-1)  \emb(R)/(p^e-1)$.

Let $(\cata{R}, \cata{\m})$ be the catapower of $(R, \m)$, $\cata{\D}$ be the flat pullback of $\D$ to $\Spec \cata{R}$ and $\fa_\infty$ be the ideal $\catae{\fa_m} \subseteq \cata{R}$.
It follows from Lemma \ref{ulim prod} that for every integers $u, n \ge 0$ we have 
\[
\newtau{n}{u}(\cata{R}, \cata{\D}, \fa_\infty^{l \cdot t_0}) = \catae{ \newtau{n}{ u}(R, \D, \fa_m^{l \cdot t_0})}.
\]

By Proposition \ref{disc rat}, there exists an integer $n_0 \ge 0$ such that 
\[
\tau ( \cata{R}, \cata{\D}, \fa_\infty^{\qadic{l \cdot t_0}{n_0}}) = \tau(\cata{R}, \cata{\D}, \fa_\infty^{\qadic{l \cdot t_0}{(n_0+1)}})
\]

On the other hand, by Proposition \ref{upper test basic} (8), there exists an integer $u_0$ such that for every integers $u \ge u_0$ and $n \ge 0$, we have
\[
\newtau{ n}{ u}(\cata{R}, \cata{\D}, \fa_\infty^{l \cdot t_0 }) = \tau(\cata{R}, \cata{\D}, \fa_\infty^{\qadic{l \cdot t_0}{n}}).
\]
Therefore, we have
\[
\catae{\newtau{ n_0}{ u} (R, \D, \fa_m^{l \cdot t_0})} = \catae{ \newtau{n_0+1}{  u} (R, \D, \fa_m^{l \cdot t_0})} \subseteq \cata{R}.
\]

Since $\m^{M_0} \cdot \tau(R, \D) \subseteq R$ is an $\m$-primary ideal, it follows from Lemma \ref{ultra incl} that there exists $T \in \U$ such that for every $m \in T$, we have 
\[
\newtau{ n_0}{u} (R, \D, \fa_m^{l \cdot t_0}) \subseteq \newtau{ n_0+1}{ u} (R, \D, \fa_m^{l \cdot t_0}) + \m^{M_0} \cdot \tau(R, \D),
\]
which completes the proof.

\end{proof}

\begin{thm}[Main Theorem]\label{main}
Let $(X=\Spec R, \D)$ be a pair such that $\tau(R, \D)$ is $\m$-primary or trivial and that $(p^e-1)(K_X+\D)$ is Cartier for some integer $e > 0$, and let $I \subseteq R$ be an $\m$-primary ideal.
Then, the set
\[
\FJN^I(R,\D) : = \left\{ \fjn^I (R, \D; \fa) \mid \fa \subsetneq R \right\}
\]
satisfies the ascending chain condition.
\end{thm}

\begin{proof}
We assume the contrary.
Then there exists a family of ideals $\{ \fa_m \}_{m \in \N}$ such that $\{ \fjn^I(R, \D; \fa_m) \}_{m \in \N}$ is a strictly ascending chain.
Set $t := \lim_{m \to \infty} \fjn^I (R, \D; \fa_m)$.
It follows from Proposition \ref{disc rat} and Theorem \ref{fpt sh2} that $t \in \Q_{>0}$.

Let $\U$ be a non-principal ultrafilter, $\cata{R}$ be the catapower of $R$, $\cata{\D}$ be the flat pullback of $\D$ to $\Spec \cata{R}$, and $\fa_\infty := \catae{\fa_m} \subseteq \cata{R}$.
Take elements $f_1, \dots, f_l \in \cata{R}$ such that $\fa_\infty = (f_1, \dots, f_l)$.
Since the natural map $\prod_{m \in \N} \fa_m \to \catae{\fa_m}$ is surjective, there exists $f_{m, i} \in \fa_m$ for every $m \in \N$ such that $f_i =\catae{f_{m,i}}$.

Set $\fa_m' : = (f_{m,1}, \dots, f_{m,l}) \subseteq \fa_m$.
Since we have $\catae{\fa_m' } = \fa_\infty$, it follows from Theorem \ref{fpt sh2} that $\sh(\ulim_m \fjn^I (R, \D; \fa_m'))=t$.
On the other hand, since we have $\fjn^I (R, \D; \fa_m^{\prime}) \le \fjn^I (R, \D; \fa_m) <t$, by replacing by a subsequence, we may assume that the sequence $\{ \fjn^I (R, \D; \fa_m') \}$ is a strictly ascending chain.
By replacing $\fa_m$ by $\fa_m'$, we may assume $\mu_R (\fa_m) \le l$ for every $m$.

By enlarging $e$, we may assume that $q=p^e$ satisfies the following properties:
\begin{enumerate}
\item $q(q-1)t \in \N$ and
\item $q>\lul+ l+\emb(R)$.
\end{enumerate}

It follows from Proposition \ref{CondA holds} that there exist integers $u,n_1>0$ and $T \in \U$ such that 
\[
\newtau{n}{u}(R, \D, \fa_m^t) \subseteq I \textup{ if and only if } \newtau{n_1}{u}(R, \D, \fa_m^t) \subseteq I
\]
for every integer $n \ge n_1$ and $m \in T$.
By enlarging $u$, we may further assume that $u \ge \ustab(\cata{R}, \cata{\D} , \fa_\infty; e)$

For every $m \in \N$ and for every sufficiently large $n \gg 0$ we have 
\[
\newtau{ n}{ u} (R, \D, \fa_m^t) \subseteq \tau ( R, \D, \fa_m^{\qadic{t}{n}}) \subseteq I.
\]
Therefore we have $\newtau{ n_1}{ u}(R, \D, \fa_m^t) \subseteq I$ for every $m \in T$.

On the other hand, since $\qadic{t}{n_1} < t = \fjn^{I\cdot \cata{R}}(\cata{R}, \cata{\D} ;\fa_\infty)$, we have 
\begin{eqnarray*}
\catae{ \newtau{ n_1}{u} (R, \D, \fa_m^t)} &=& \newtau{ n_1}{u} (\cata{R}, \cata{\D}, \fa_\infty^t) \\
&=& \tau (\cata{R}, \cata{\D}, \fa_\infty^{\qadic{t}{n_1}}) \\
& \not\subseteq & I \cdot \cata{R}.
\end{eqnarray*}
Therefore, there exists a set $S \in \U$ such that
\[
\newtau{ n_1}{ u} (R, \D, \fa_m^t) \not\subseteq I
\]
for every $m \in S$.
Since $S \cap T \neq \emptyset$, we have contradiction.

\end{proof}

\begin{cor}[Theorem \ref{intro reg}]\label{reg ACC}
Fix an integer $n \ge 1$, a prime number $p>0$ and a set $\Dreg{n}{p}$ such that every element of $\Dreg{n}{p}$ is an $n$-dimensional $F$-finite Noetherian regular local ring of characteristic $p$.
The set
\[
\T^{\mathrm{reg}}_{n,p}: = \{ \mathrm{fpt} (A; \fa) \mid A \in \Dreg{n}{p} ,\fa \subsetneq A \},
\]
 satisfies the ascending chain condition.
\end{cor}

\begin{proof}
We assume the contrary.
Then there exists a sequence $\{ A_m \}_{m \in \N}$ in $\T^{\mathrm{reg}}_{n,p}$ and ideals $\fa_m \subsetneq A_m$ such that the sequence $\{ \fpt(A_m; \fa_m) \}$ is a strictly ascending chain.

Since test ideals commute with completion(\cite[Proposition 3.2]{HT}), we may assume that $A_m = k_m[[x_1, \dots, x_n]]$ for some $F$-finite field $k_m$.
Take an $F$-finite field $k$ such that $k_m \subseteq k$ for every $m$.
Let $(A,\m_A)$ be the local ring $k[[x_1, \dots, x_n]]$.
Then it follows as in the proof of \cite[Theorem 3.5 (i)]{BMS2} that $\fpt (A; (\fa_m A)) = \fpt(A_m; \fa_m)$.
Therefore, we have $\fpt(A_m; \fa_m) \in \FJN^{\m_A}(A, 0)$ for every $m$, which contradicts to Theorem \ref{main}.
\end{proof}

Let $(R, \m, k)$ be a Noetherian local ring of equicharacteristic.
Then $(R,\m)$ is said to be a \emph{quotient singularity} if there exist a regular affine variety $U=\Spec A$ over $k$, a finite group $G$ with a group homomorphism $G \to \Aut_k(U)$, and a point $x$ of the quotient $V=U/G := \Spec (A^G)$ such that there exists an isomorphism $\widehat{R} \cong \widehat
{\sO_{V, x}}$ as rings.
Moreover, if $|G|$ is coprime to $\chara(k)$, then we say that $(R,\m)$ is a \emph{tame quotient singularity}.

\begin{lem}\label{tame quot basic}
Let $(R,\m, k)$ be a tame quotient singularity of dimension $n$.
Then, there exists a finite group $G \subseteq \GL_n(k)$ with the following properties.
\begin{enumerate}
\item $| G |$ is coprime to $\chara(k)$.
\item The natural action of $G$ on the affine space $\A^n_k$ has no fixed points in codimension 1.
\item Let $V : = \A^n_k/G$ be the quotient and $x \in V$ be the image of the origin of $\A^n_k$.
Then we have $\widehat{R} \cong \widehat{\sO_{V,x}}$.
\end{enumerate}
\end{lem}

\begin{proof}
The proof follows as in the case when $\chara(k)=0$ (see \cite[p.15]{dFEM}), but for the convenience of reader we sketch it here.

Since $R$ is a tame quotient singularity, there exists a regular affine variety $U$, a finite group $G$ which acts on $U$ such that $|G|$ is coprime to $\chara(k)$, and a point $x \in V$ such that $\widehat{R} \cong \widehat{\sO_{V,x}}$.

Take a point $y \in U$ with image $x$.
By replacing $G$ by the stabilizer subgroup $G_y \subseteq G$, we may assume that $G$ acts on the regular local ring $(A, \m_A) : = (\sO_{U,y}, \m_y)$.
Since $|G|$ is coprime to $\chara(k)$, it follows from Maschke's theorem that the natural projection $\m_A \to \m_A/\m_A^2$ has a section as $k[G]$-modules.
This section induces $k[G]$-algebra homomorphism $\Gr_{\m_A}(A) \to A$, where $\Gr_{\m_A}(A)$ is the associated graded ring of $(A, \m_A)$.
Therefore, by replacing $U$ by $\Spec(\Gr_{\m_A}(A))$, we may assume that $U=\A^n_k$ and $G \subseteq \GL_n(k)$.

Let $H \subseteq G$ be the subgroup generated by elements $g \in G$ which fixes some codimension one point of $U$.
Since $|G|$ is coprime to $\chara(k)$, it follows from Chevalley–-Shephard-–Todd theorem (see for example \cite[Theorem 7.2.1]{Ben}) that $U/H \cong \A^n_k$.
By replacing $U$ by $U/H$ and $G$ by $G/H$, we complete the proof of the lemma.
\end{proof}

\begin{prop}[Theorem \ref{intro quot}]\label{quot ACC}
Fix an integer $n \ge 1$, a prime number $p>0$ and a set $\Dquot{n}{p}$ such that every element of $\Dquot{n}{p}$ is an $n$-dimensional $F$-finite Noetherian normal local ring of characteristic $p$ with tame quotient singularities.
The set 
\[
\T^{\mathrm{quot}}_{n,p}: = \{ \mathrm{fpt} (R; \fa) \mid R \in \Dquot{n}{p}, \fa \subsetneq R \textup{ is an ideal} \}
\]
 satisfies the ascending chain condition.
\end{prop}

\begin{proof}
The proof is essentially the same as \cite[Proposition 5.3]{dFEM}.
Let $(R,\m,k)$ be a local ring such that $R \in \Dquot{n}{p}$ and $\fa \subsetneq R$ be an ideal of $R$.
Let $G$, $V$, and $x$ be as in Lemma \ref{tame quot basic}.
Consider the natural morphism $\pi : U:= \A^n_k \to V$.
Since $G$ is a finite group, the morphism $\pi$ is a finite surjective morphism with $\deg(\pi)$ coprime to $\chara(k)$.
Since $G$ acts on $U$ with no fixed points in codimension one, the morphism $\pi$ is \'{e}tale in codimension one.

Set $W := \Spec(\widehat{R})$ and $U' : = U \times_V W$.
Since $U$ is a regular scheme and $W \to V$ is a regular morphism, each connected component of $U'$ is a regular scheme. 
Fix a connected component $U'' \subseteq U'$.

Since the morphism $\widehat{\pi}: U'' \to W$ is finite surjective, \'{e}tale in codimension 1 and $\deg{\widehat{\pi}}$ is coprime to $p$, it follows from \cite[Theorem 3.3]{HT} that 
\[
\fpt(W ;\fa \sO_W) = \fpt(U''; \fa \sO_{U''}).
\]

On the other hand, since the test ideals commute with completion (\cite[Proposition 3.2]{HT}), we have
\[
\fpt(R; \fa)=\fpt(W ; \fa \sO_W).
\]

Therefore, it follows from Corollary \ref{reg ACC} that the set $\T^{\mathrm{quot}}_{n,p}$ satisfies the ascending chain condition.

\end{proof}

We conclude with a natural question as below.

\begin{ques}
Does Theorem \ref{intro reg} give an alternative proof of \cite[Theorem 1.1]{dFEM}?
Moreover, does Theorem \ref{main} imply that the set of all jumping numbers of multiplier ideals with respect to a fixed $\m$-primary ideal on a log $\Q$-Gorenstein pair over $\C$ satisfies the ascending chain condition?
\end{ques}

We hope to consider this question at a later time.
%%%%%%%%%%%%%%%%%%%%%%%%%%%%%%%%%%%%%%%%%%%%%%%%%%%%%%%%%%%%%%%%%%%%%%
%%%%%%%%%%%%%%%%%%%%%%%%%%%%%%%%%%%%%%%%%%%%%%%%%%%%%%%%%%%%%%%%%%%%%%
%%%%%%%%%%%%%%%%%%%%%%%%%%%%%%%%%%%%%%%%%%%%%%%%%%%%%%%%%%%%%%%%%%%%%%
%%%%%%%%%%%%%%%%%%%%%%%%%%%%%%%%%%%%%%%%%%%%%%%%%%%%%%%%%%%%%%%%%%%%%%

%List of mathematics journals

%[In Jurna]
%\bibitem[NameYear]{label}
%Name, Title, Jurnal \textbf{Number} (Year) no. *, Page--Page.
%MLM in Google scholor minus ""

%[On arXiv]
%\bibitem[NameYear]{label}
%Name, Title, arXiv:Num, preprint (year).

%[Book]
%\bibitem[NameYear]{label}
%Name, Italic{Title}, SeriesName, Syuppansya, Basho, year.

%[Chapter in a Book] (cf.Gabber no ano book)
%\bibitem[NameYear]{label}
%Name, Title, Italic{Title of the book}, pp. ??--??,  (SeriesName), Syuppansya, Basho, year.

\end{document}